\documentclass[12pt]{article}
\usepackage[utf8]{inputenc}
\usepackage[T1]{fontenc}
\usepackage[top=25mm, bottom=25mm, left=25mm, right=25mm]{geometry} 
\usepackage{amsfonts,amsmath,amsthm,amssymb}
\usepackage{hyperref,enumitem,enumerate}
\hypersetup{
    colorlinks=true,
    urlcolor=black,
}
\usepackage{graphics}
\usepackage{tikz}
\usepackage{xcolor}
\usetikzlibrary{patterns, patterns.meta}
\usetikzlibrary{calc}

\definecolor{applegreen}{rgb}{0.55, 0.71, 0.0}
\definecolor{brilliantrose}{rgb}{1.0, 0.33, 0.64}
\definecolor{gptblue}{rgb}{0.20, 0.55, 0.85}
\definecolor{gptorange}{rgb}{0.90, 0.42, 0.08}
\definecolor{gptpurple}{rgb}{0.48, 0.24, 0.72}
\definecolor{gptteal}{rgb}{0.00, 0.55, 0.52}
\definecolor{gptgold}{rgb}{0.85, 0.62, 0.05}

\newtheorem{theorem}{Theorem}

\newtheorem{corollary}[theorem]{Corollary}

\newtheorem{observation}[theorem]{Observation}
\newtheorem{claim}[theorem]{Claim}
\newtheorem{conjecture}[theorem]{Conjecture}

\theoremstyle{definition}

\usepackage{mathtools}

\newcommand{\cF}{\mathcal{F}}

\newcommand{\bg}{\hat{B}}

\DeclareMathOperator{\sm}{sum}

\let\epsilon=\varepsilon
\let\eps=\epsilon

\begin{document}

\title{Rainbow Tur\'an numbers for paths of length four}

\author{
Sylwia Antoniuk\thanks{Faculty of Mathematics and Computer Science, Adam Mickiewicz University, Uniwersytetu Poznańskiego 4, 61-614  Pozna\'n, Poland. E-mail: {\tt antoniuk@amu.edu.pl}. Supported by the National Science Centre grant 2024/53/B/ST1/00164.}\and
Andrzej Grzesik\thanks{Faculty of Mathematics and Computer Science, Jagiellonian University, {\L}ojasiewicza 6, 30-348 Krak\'{o}w, Poland. E-mail: {\tt andrzej.grzesik@uj.edu.pl}. Supported by the National Science Centre grant 2021/42/E/ST1/00193.}\and
Magdalena Prorok\thanks{AGH University of Krakow, al.~Mickiewicza 30, 30-059 Kraków, Poland. E-mail: {\tt prorok@agh.edu.pl}. Supported by the National Science Centre grant 2021/42/E/ST1/00193.}\and
Andrzej Ruciński\thanks{Faculty of Mathematics and Computer Science, Adam Mickiewicz University, Uniwersytetu Poznańskiego 4, 61-614 Pozna\'n, Poland. E-mail: {\tt rucinski@amu.edu.pl}. Supported by the National Science Centre grant 2024/53/B/ST1/00164.}}

\date{}

\maketitle

\begin{abstract}
 Given a set $V$ of $n$ vertices and an integer $k\ge1$, our goal is to maximize the number of edges in graphs $G_1, G_2, \ldots, G_k$, defined on $V$, under the constraint that the union of all graphs, thought of as a multi-graph, does not contain a~\emph{rainbow copy} of the path $P_5$ on $5$ vertices, that is, a copy of $P_5$ with each of its four edges belonging to a~different~$G_i$. 
 
 We consider two versions of the problem, in which, respectively, $\sum_i e(G_i)$ and $\min_i e(G_i)$ are maximized. In the former case, we determine the maximum precisely for all $k\le n-1$ (and also for $P_4$). In the latter, we obtain an asymptotic value for $k\in\{5,6,9\}$ and formulate a very plausible conjecture for all other values of $k$. 
 We also solve the problem for $k=4$, but under an additional assumption of completeness.
\end{abstract}

\section{Introduction}

Given $k$ graphs $G_1,\dots,G_k$ on the same vertex set $V$, we color edges of each graph with a~different color, and define the \emph{colored multi-graph} $M:=M(G_1,\dots,G_k)$ as a multi-graph on $V$ with all edges colored, in which, for every $i \in [k]$ and $u,v \in V$, there is exactly one edge $uv$ in color $i$ whenever $uv \in E(G_i)$. 
For better exposition, we will visualize each multi-edge of $M$ as a set of parallel edges of distinct colors.

A subgraph $F$ of $M$ will be called \emph{rainbow} if each edge of $F$ has a different color, that is, it belongs to a different graph~$G_i$.
For a family of multi-graphs $\cF$ we say that a colored multi-graph $M$ is \emph{rainbow} $\cF$-free if $M$ does not contain a rainbow subgraph isomorphic to any $F\in\cF$.

Following, respectively, \cite{KSSV04} and \cite{DM}, we define two natural \emph{rainbow} Tur\'an-type numbers. Given an integer $k\ge1$ and a set of multi-graphs $\cF$, let
$$ex_k^{\sm}(n,\cF)=\max\left\{\sum_{i=1}^k|E(G_i)|: M(G_1, \ldots, G_k) \ \mbox{is a rainbow $\cF$-free multi-graph on $[n]$}\right\}$$
and
$$ex_k^{\min}(n,\cF)=\max\left\{\min_{1\le i\le k}|E(G_i)|: M(G_1, \ldots, G_k) \ \mbox{is a rainbow $\cF$-free multi-graph on $[n]$}\right\}\!.$$

If $\cF=\{F\}$, then we say $F$-free instead of $\{F\}$-free, and write $ex_k^{\sm}(n,F)$ and $ex_k^{\min}(n,F)$ instead of $ex_k^{\sm}(n,\{F\})$ and $ex_k^{\min}(n,\{F\})$.

Let us mention that in \cite{Frankl-et-al} the study of yet another version of a rainbow Tur\'an-type number has been initiated, where instead the product $\prod_{i=1}^ke(G_i)$ is maximized, while in~\cite{FMR24} other possible measures of rainbow $K_3$-free 3-colored multi-graphs were considered.  A summary of results on the topic can be found in a recent survey \cite{survey}.

\subsection{Results on the $\sm$ version}\label{Isum}

We first focus on the rainbow Tur\'an numbers $ex_k^{\sm}(n,\cF)$. For these numbers, there is an obvious lower bound in terms of the standard Tur\'an number,
    \[ex(n,\cF)=\max\left\{|E(G)|:\;G \mbox{ is an $\cF$-free graph on $[n]$}\right\}.\]
Indeed, for any $k\ge1$,
    \begin{equation}\label{Turan_bound}
    ex_k^{\sm}(n,\cF)\ge k\cdot ex(n,\cF),
    \end{equation}
by taking $G_i=T(n,\cF)$, $i=1,\dots,k$, where $T(n,\cF)$ is a fixed $\cF$-free graph with $ex(n,\cF)$ edges and vertex set $[n]$. For $k\ge \ell:=\min_{F\in\cF}|E(F)|$, there is a competing lower bound
    \begin{equation}\label{clique_bound}
    ex_k^{\sm}(n,\cF)\ge(\ell-1)\binom n2,
    \end{equation}
obtained by taking $G_i=K_n$, $i=1,\dots,\ell-1$, and $G_i=\emptyset$ otherwise. Note that for $k<\ell$, trivially $ex_k^{\sm}(n,\cF)=k\binom n2$.

In  \cite{KSSV04}, Keevash, Saks, Sudakov, and Verstra\"ete proved that for a broad class of graphs~$F$, either of these two lower bounds determines, in fact, the value of $ex_k^{\sm}(n,F)$. In particular, it holds for cliques $K_r$, $r\ge3$:

\begin{equation}\label{eq:cliquebound}
ex_k^{\sm}(n,K_r)
    =\begin{cases}
    \left(\binom r2-1\right)\binom n2 & \mathrm{if}\quad \binom r2\le k\le (r^2-1)/2\,,\\
    k\cdot ex(n,K_r) & \mathrm{if}\quad k \geq(r^2-1)/2\,.
    \end{cases}
\end{equation}

They also showed a similar result for 3-color-critical graphs and conjectured it for $r$-color-critical graphs for all $r\ge3$. Chakraborti, Kim, Lee, Liu, and Seo \cite{CKLLS} proved it for 4-color-critical graphs and for almost every $r$-color-critical graph for $r \geq 5$, while Li, Ma, and Zheng \cite{LMZ} proved that indeed $ex_k^{\sm}(n,F) = k\cdot ex(n,F)$ for $r$-color-critical graph $F$ with $k$ large enough (in terms of $r$ and the number of edges of $F$). 

Keevash, Saks, Sudakov, and Verstra\"ete \cite{KSSV04} considered also bipartite graphs, which, in the context of Tur\'an numbers, are the most difficult, and proved partial results for a selection of graphs, including the 4-cycle $C_4$: there exist constants $0<c<C$ such that
   \[ex_k^{\sm}(n,C_4)
    =\begin{cases}
     3\binom n2 & \mathrm{if}\quad 4\le k\le c\sqrt n\,,\\
    k\cdot ex(n,C_4) & \mathrm{if}\quad k >C\sqrt n\,.
    \end{cases}\]
So, in this case, the two canonical constructions, \eqref{Turan_bound} and \eqref{clique_bound}, still determine the rainbow ``sum'' Tur\'an number, but only in the initial and final range of $k$, leaving the middle unsettled.

 \smallskip

Our first two results confirm this phenomenon for almost the full range of $k$ for two other small bipartite graphs, the paths of lengths three and four.
Recall that Faudree and Schelp \cite{FSch75} and, independently, Kopylov \cite{K77}, improving upon the celebrated result of Erd\H os and Gallai \cite{EG59}, proved that for all $n\ge \ell+1\ge 3$,
$$ex(n,P_{\ell+1})=\frac{(\ell-1)(n-r)}2+\binom r2,$$
where $r\equiv n \pmod \ell$, and the optimal construction consists of $(n-r)/\ell$ vertex disjoint cliques $K_\ell$ and, if $r>0$, one clique $K_r$.
In particular,
    \[
    ex(n,P_4)=
    \begin{cases}
        n & \mathrm{if}\quad n\equiv 0 \pmod 3\,,\\
        n-1 & \mathrm{if}\quad n\not\equiv 0 \pmod 3\,,
    \end{cases}
    \]
while 
    \[
    ex(n,P_5)=
    \begin{cases}
        3n/2 & \mathrm{if}\quad n\equiv 0 \pmod 4\,, \\
        3(n-1)/2 & \mathrm{if}\quad n\equiv 1,3 \pmod 4\,, \\
        (3n-4)/2 & \mathrm{if}\quad n\equiv 2 \pmod 4\,.
    \end{cases} 
    \]
It should be noted that for $P_4$ and $n\not\equiv 0 \pmod 3$ there is another extremal graph, namely a~star $K_{1,n-1}$.

The non-divisible cases cause some problems when working on the rainbow versions.
Therefore, the two results, which we state below, are not fully exact, but are tight up to a small additive constant.

\begin{theorem}\label{thm:p4sum}
    For all integers $n$ and $k\geq3$,
    \[ex_k^{\sm}(n,P_4)
    \begin{cases}
    =2\binom n2 & \mathrm{if}\quad k\le n-1\,,\\
    =kn & \mathrm{if}\quad k\ge n,\ n\equiv 0 \pmod 3 \,,\\
    \le kn & \mathrm{if}\quad k\ge n,\ n\not\equiv 0 \pmod 3\,.
    \end{cases}\]
\end{theorem}

\noindent Note that in the first case the lower bound is provided by~\eqref{clique_bound}, while in the second case by~\eqref{Turan_bound}. For  $k\geq n,\;n\not\equiv 0 \pmod 3$,~\eqref{Turan_bound} yields $ex_k^{\sm}(n,P_4)\ge k(n-1)$ that almost matches the upper bound stated above.

\begin{theorem}\label{thm:p5sum}
    For all integers $n$ and $k\geq4$,
    \[ex_k^{\sm}(n,P_5)
    \begin{cases}
    =3\binom n2 & \mathrm{if}\quad k \le n-1 \,,\\
    =\frac32kn & \mathrm{if}\quad k\ge n,\ n\equiv 0 \pmod 4 \,,\\
   \le \frac32kn & \mathrm{if}\quad k\ge n,\ n\not\equiv 0 \pmod 4 \,. \end{cases}\]
\end{theorem}
\noindent Again,  in the first case the lower bound is provided by~\eqref{clique_bound}, while in the second case by~\eqref{Turan_bound}.  For $k\geq n,\ n\not\equiv 0 \pmod 4$,~\eqref{Turan_bound} yields lower bounds $ex_k^{\sm}(n,P_5)\ge\frac32k(n-1)$ when $n \equiv 1,3 \pmod 4$, and $ex_k^{\sm}(n,P_5)\ge\frac12k(3n-4)$ when $n\equiv 2 \pmod 4$, both very close to the upper bound.

\subsection{Results on the $\min$ version}\label{res_min}

In this variant, the construction behind the lower bound~\eqref{Turan_bound} still stands, yielding the bound
$ex_k^{\min}(n,\cF)\ge ex(n,\cF)$. However, for bipartite connected graphs, it is grossly improved by the following general construction.

Throughout this subsection, we
restrict ourselves to singleton families $\cF=\{F\}$, and further assume that $F$ is a~connected multi-graph with $\ell$ edges.  Setting 
$$\kappa:=\left\lceil\frac k{\ell-1}\right\rceil,$$ one may split $V$ into $\kappa$ disjoint sets $V_1,\dots,V_\kappa$, $\lfloor n/\kappa\rfloor\le |V_j|\le\lceil n/\kappa\rceil$, $j=1,\dots,\kappa$, and plant up to $\ell-1$ cliques $G_i$ on each set $V_j$, so that all graphs $G_i$ are taken care of, obtaining the lower bound
    \begin{equation}\label{split_bound}
    ex_k^{\min}(n,F)\ge \binom{\lfloor n/\kappa\rfloor}2.
    \end{equation}

As we will soon learn, for the rainbow Tur\'an numbers $ex^{\min}_k(n,F)$, so far mainly approximate results have been obtained. They are best expressed in terms of the  \emph{rainbow Tur\'an density} defined as
    $$\pi^{\min}_k(F)=\lim_{n\to\infty}\frac{ex^{\min}_k(n,F)}{\binom n2}.$$
Note that $\pi^{\min}_1(F)$ equals the ordinary Tur\'an density $\pi(F)$.

In \cite[Theorem 1.1]{IKLS25}, it was proved that for each $F$ and $k\ge 1$ the limit defining $\pi^{\min}_k(F)$ does exist. Note that for $k<\ell$, we have $\pi^{\min}_k(F)=1$ and, otherwise, $\pi^{\min}_k(F)\ge\pi(F)$.
Note also that, as a consequence of~\eqref{split_bound}, for every connected multi-graph $F$  we have
$\pi^{\min}_k(F)\ge \kappa^{-2},$
and summarizing,
    $$\pi^{\min}_k(F)\ge \max\left\{\pi(F),\,\left\lceil \frac k{\ell-1}\right\rceil^{-2}\right\}.$$

The first result in this direction dealt with triangles. By (\ref{eq:cliquebound}) we have that for any $k \geq 4$, $\pi^{\min}_k(K_3) = \pi(K_3) = 1/2$, while for $k=3$, perhaps surprisingly, Magnant \cite{Magnant} showed that the above lower bound is not tight. Later, Aharoni, DeVos, Gonz\'alez Hermosillo de la Maza, Montejano and \v S\'amal \cite[Theorem 1.2]{ADGMS20} proved that
$$\pi^{\min}_3(K_3)=\frac{52-4\sqrt7}{81}\approx0.5114.$$

Besides triangles, the only graphs $F$ for which $\pi^{\min}_k(F)$ has been determined for every $k$ are some special trees. For trees and, more generally, for all bipartite graphs, $\pi(F)=0$, and thus
    \begin{equation}\label{kell2}
    \pi^{\min}_k(F)\ge \lceil  k/(\ell-1)\rceil^{-2}.
    \end{equation}
This was complemented by  Im, Kim, Lee, and Seo in
 \cite[Proposition 1.11]{IKLS25} with an upper bound
$\pi^{\min}_k(F)\le(\ell-1)/k$. They also showed in \cite[Theorem 1.3]{IKLS25} that among all trees $T_{k+1}$ on $k+1$ vertices the rainbow Tur\'an density (with $k$ colors) is maximized by stars and determined the latter:

\[\pi_k^{\min}(T_{k+1})\le \pi_k^{\min}(K_{1,k})=\left(\frac{k-1}k\right)^2.\]
Note that this is a better bound than $(\ell-1)/k$, since for $T_{k+1}$ we have $\ell=k$.

The last known case so far is that of the path $P_4$, proved by Babi\'nski and the second author in \cite[Theorem 1]{BG23}.

\begin{theorem}[\cite{BG23}]\label{thm:p4min} We have
    \[
    \pi_k^{\min}(P_4)=\begin{cases}
    \left\lceil\frac k2\right\rceil^{-2}& \mathrm{if}\quad 3\le k\le 6 \,,\\
    \frac1{2k-1} & \mathrm{if}\quad k\ge7\,.\\
    \end{cases}
    \]
\end{theorem}
\noindent Note that for $3\le k\le 6$ the value of $\pi_k^{\min}(P_4)$ is equal to the lower bound~\eqref{kell2}, while for $k\ge7$ it is larger, e.g., for $k=8$ it is $\frac1{15}$ versus $\frac1{16}$.
For a relevant construction for $k\ge7$, adapted to $P_5$, see Section~\ref{conjk}.

In this paper, we prove an analogous result for $P_5$, the path of length four, but  only for a~few values of $k$.

\begin{theorem}\label{thm:p5min} We have
    \[
        \pi_k^{\min}(P_5)=
        \begin{cases}
        \frac14 & \mathrm{if}\quad k\in\{5,6\} \,,\\
        \frac19 & \mathrm{if}\quad k=9\,.\\
    \end{cases}\]
\end{theorem}
\noindent Note that the obtained values coincide with the lower bound~\eqref{kell2}.

In Section~\ref{conjk} we propose a conjecture determining the value of $\pi_k^{\min}(P_5)$ for all remaining values of $k$. The most interesting case of $k=4$ also turned out to be the most challenging one. 
We conjecture in Section~\ref{conj4}, based on a sophisticated lower bound construction, that $\pi_4^{\min}(P_5)=153-108 \sqrt{2}\approx 0.265$. Thus, again, we would get a bound higher than that in~\eqref{kell2}.

As the final result of this paper, we precisely determine a (weaker) variant of the Tur\'an number $ex_4^{\min}(n,P_5)$.
Namely, we restrict ourselves to instances of colored multi-graphs $M(G_1,G_2,G_3,G_4)$ with $E(G_1)\cup E(G_2)\cup E(G_3)\cup E(G_4)=\binom V2$. We will call such colored multi-graphs \emph{complete}. 

Formally, by $\overline{ex}_k^{\min}(n,\cF)$ we denote
$$\max\left\{\min_{1\le i\le k}|E(G_i)|: M(G_1, \ldots, G_k) \ \mbox{is a rainbow $\cF$-free complete multi-graph on $[n]$}\right\}.$$ Note that $\overline{ex}_k^{\min}(n,\cF)\le ex_k^{\min}(n,\cF)$.

In \cite{Magnant} Magnant determined the limit of $\overline{ex}_3^{\min}(n,K_3)$, which later, in \cite{ADGMS20}, was proved to match the limit of $ex_3^{\min}(n,K_3)$. 
In the case of $P_5$, however, the complete version yields a~much smaller threshold. In this case, we were able not only to find the limit, but also to pinpoint the Tur\'an number precisely.

\begin{theorem}\label{thm:P5_complete}
For any $n\ge 21$,
\[ \overline{ex}_4^{\min}(n,P_5)=
        \begin{cases}
        \binom{\lfloor n/3 \rfloor}2 & \mathrm{if}\quad n \not\equiv 2 \pmod 3\,,\\
        \binom{\lfloor n/3 \rfloor}2 + 1 & \mathrm{if}\quad n \equiv 2 \pmod 3\,.\\
        \end{cases}
\]
\end{theorem}

An optimal construction, yielding the lower bound for $n\ge 21$, is obtained by splitting the vertex set into three equal sets $A_i$ of size $\lfloor n/3 \rfloor$, for $i \in\{1,2,3\}$, each forming a clique in $G_i$, and taking $G_4=K_n$.
If $n\equiv 2 \pmod 3$, then one can additionally connect the remaining two vertices by edges in each graph $G_i$. 
This gives at least $\binom{\left\lfloor\frac n3\right\rfloor}2\sim\frac{1}{18}n^2$ edges in each of $G_1$, $G_2$, and $G_3$, while $e(G_4)=\binom n2$. 
However, to the same effect, one can take as $G_4$ any graph containing the complete multipartite graph with partition $A_1$, $A_2$, $A_3$ and $V \setminus (A_1 \cup A_2 \cup A_3)$.

The assumption $n\ge 21$ in Theorem~\ref{thm:P5_complete} is needed because the cases $n<21$ may have a~different bound. For example, for $n=20$ by considering each graph $G_1$, $G_2$, $G_3$ to be the star $K_{1,n-1}$ centered at the same vertex and $G_4=K_n$, we have at least $n-1=19$ edges in each color, while Theorem~\ref{thm:P5_complete} suggests only $\binom{6}{2}+1 = 16$ edges. 

\subsection{Notation} 

For a colored multi-graph $M:=M(G_1, \ldots, G_k)$ on a set $V(M)$, let $e(M)=\sum_{i=1}^ke(G_i)$, where \hbox{$e(G_i)=|E(G_i)|$}. For a vertex $v\in V(M)$ and $i\in[k]$, let $d_i(v)$ denote the degree of $v$ in $G_i$ and $d(v) = d_1(v) + \ldots + d_k(v)$ -- the total degree of $v$ in $M$. Further, let $N_i(v)$ be the neighborhood of $v$ in graph $G_i$ and $N(v) = \bigcup_{i \in [k]} N_i(v)$. The \emph{multiplicity} of an edge $e = uv$ in $M$, denoted by $m(e)$ or $m(u,v)$, is defined as the number of graphs $G_i$ for which $uv \in E(G_i)$. 

The following notation will be used in Sections~\ref{Max_min} and~\ref{complete}.
For a graph $G$, let 
$$\overline{V}(G)=\{v\in V(G): d(v)>0\}$$
be the set of all non-isolated vertices of $G$. Given a $k$-colored multi-graph $M$ and a vertex $v\in V(M)$, let $\overline{V}(M) = \bigcup_{i=1}^{k}\overline{V}(G_i)$ and $Spe(v)=\{i\in[k]: v\in \overline{V}(G_i)\}$ be the \emph{spectrum} of~$v$. 
Set also $spe(v)=|Spe(v)|$ and note that $i\in Spe(v)$ if and only if $d_i(v)\ge1$.

\subsection{Homomorphisms} 

The standard Tur\'an number of a graph $F$ is equal, up to a small $o(n^2)$ term, to the Tur\'an number of the family of all homomorphic images of $F$, which is usually easier to determine. As observed in~\cite{IKLS25}, a~similar phenomenon holds for rainbow Tur\'an numbers.

For a given multi-graph $F$, let $Hom(F)$ be the set of all pairwise non-isomorphic multi-graphs $H$, for which there is a mapping $h:V(F)\to V(H)$ such that for every edge $uv\in E(F)$, its image  $h(u)h(v)\in E(H)$, and for different edges of $F$ their images are different edges of~$H$ (such multi-graphs $H$ are called in~\cite[Definition 2.2]{IKLS25} edge-preserving homomorphic images of $F$).

For $F=P_4$ there are four, and for $F=P_5$ there are nine multi-graphs in $Hom(F)$ (see Figure \ref{Fig:path-homomorphic-images}), which correspond to the walks of length, respectively, three and four, with multiple edges whenever the walk passes through an edge multiple times. 
In Sections~\ref{Max_min} and~\ref{complete} we denote by $W_4$ any (unspecified) element of $Hom(P_5)$, that is, any walk of length four. Analogously, by $W_3$ we denote any element of $Hom(P_4)$. 

\begin{figure}[ht]

\begin{center}
\begin{tikzpicture}[scale=0.9]
    \node at (-1,1) {$W_3$};

    \filldraw [black] (0,0) circle (2pt);
    \filldraw [black] (1,0) circle (2pt);
    \filldraw [black] (2,0) circle (2pt);
    \filldraw [black] (3,0) circle (2pt);
    \draw[thick] (0,0) -- (3,0);

    \filldraw [black] (5,0) circle (2pt);
    \filldraw [black] (6.5,0) circle (2pt);
    \draw[thick] (5,0) -- (6.5,0);
    \draw[thick] (5,0) .. controls (5.5,0.5) and (6,0.5) .. (6.5,0);
    \draw[thick] (5,0) .. controls (5.4,-0.5) and (6.1,-0.5) .. (6.5,0);
    
    \filldraw [black] (8.5,0) circle (2pt);
    \filldraw [black] (9.5,0) circle (2pt);
    \filldraw [black] (11,0) circle (2pt);
    \draw[thick] (8.5,0) -- (9.5,0);
    \draw[thick] (9.5,0) .. controls (10,0.5) and (10.5,0.5) .. (11,0);
    \draw[thick] (9.5,0) .. controls (10,-0.5) and (10.5,-0.5) .. (11,0);

    \filldraw [black] (12.8,-0.6) circle (2pt);
    \filldraw [black] (13.6,0.6) circle (2pt);
    \filldraw [black] (14.4,-0.6) circle (2pt);
    \draw[thick] (12.8,-0.6) -- (13.6,0.6) -- (14.4,-0.6) -- (12.8,-0.6);

    \node at (-1,-2) {$W_4$};
    \filldraw [black] (0,-3) circle (2pt);
    \filldraw [black] (1,-3) circle (2pt);
    \filldraw [black] (2,-3) circle (2pt);
    \filldraw [black] (3,-3) circle (2pt);
    \filldraw [black] (4,-3) circle (2pt);
    \draw[thick] (0,-3) -- (4,-3);

    \filldraw [black] (6,-3) circle (2pt);
    \filldraw [black] (7.5,-3) circle (2pt);
    \draw[thick] (6,-3) .. controls (6.5,-2.8) and (7,-2.8) .. (7.5,-3);
    \draw[thick] (6,-3) .. controls (6.5,-2.4) and (7,-2.4) .. (7.5,-3);
    \draw[thick] (6,-3) .. controls (6.5,-3.6) and (7,-3.6) .. (7.5,-3);
    \draw[thick] (6,-3) .. controls (6.5,-3.2) and (7,-3.2) .. (7.5,-3);

    \filldraw [black] (9.5,-3) circle (2pt);
    \filldraw [black] (10.5,-3) circle (2pt);
    \filldraw [black] (12,-3) circle (2pt);
    \draw[thick] (9.5,-3) -- (12,-3);
    \draw[thick] (10.5,-3) .. controls (11,-2.5) and (11.5,-2.5) .. (12,-3);
    \draw[thick] (10.5,-3) .. controls (11,-3.5) and (11.5,-3.5) .. (12,-3);

    \filldraw [black] (0,-5) circle (2pt);
    \filldraw [black] (1,-5) circle (2pt);
    \filldraw [black] (2,-5) circle (2pt);
    \filldraw [black] (3.5,-5) circle (2pt);
    \draw[thick] (0,-5) -- (2,-5);
    \draw[thick] (2,-5) .. controls (2.5,-4.5) and (3,-4.5) .. (3.5,-5);
    \draw[thick] (2,-5) .. controls (2.5,-5.5) and (3,-5.5) .. (3.5,-5);

    \filldraw [black] (5.5,-5) circle (2pt);
    \filldraw [black] (7,-5) circle (2pt);
    \filldraw [black] (8.5,-5) circle (2pt);
    \draw[thick] (5.5,-5) .. controls (6,-4.5) and (6.5,-4.5) .. (7,-5);
    \draw[thick] (5.5,-5) .. controls (6,-5.5) and (6.5,-5.5) .. (7,-5);
    \draw[thick] (7,-5) .. controls (7.5,-4.5) and (8,-4.5) .. (8.5,-5);
    \draw[thick] (7,-5) .. controls (7.5,-5.5) and (8,-5.5) .. (8.5,-5);

    \filldraw [black] (10.3,-5.75) circle (2pt);
    \filldraw [black] (11.3,-5.75) circle (2pt);
    \filldraw [black] (12.3,-5.75) circle (2pt);
    \filldraw [black] (11.3,-4.25) circle (2pt);
    \draw[thick] (10.3,-5.75) -- (12.3,-5.75);
    \draw[thick] (11.3,-4.25) .. controls (10.8,-4.6) and (10.8,-5.4) .. (11.3,-5.75);
    \draw[thick] (11.3,-4.25) .. controls (11.8,-4.6) and (11.8,-5.4) .. (11.3,-5.75);

    \filldraw [black] (0.5,-8.1) circle (2pt);
    \filldraw [black] (1.3,-6.9) circle (2pt);
    \filldraw [black] (2.1,-8.1) circle (2pt);
    \draw[thick] (0.5,-8.1) -- (1.3,-6.9) -- (2.1,-8.1);
    \draw[thick] (0.5,-8.1) .. controls (0.9,-7.6) and (1.7,-7.6) .. (2.1,-8.1);
    \draw[thick] (0.5,-8.1) .. controls (0.9,-8.6) and (1.7,-8.6) .. (2.1,-8.1);

    \filldraw [black] (4.6,-7.7) circle (2pt);
    \filldraw [black] (5.6,-7.7) circle (2pt);
    \filldraw [black] (6.8,-6.9) circle (2pt);
    \filldraw [black] (6.8,-8.5) circle (2pt);
    \draw[thick] (4.6,-7.7) -- (5.6,-7.7) -- (6.8,-6.9) -- (6.8,-8.5) -- (5.6,-7.7);
    
    \filldraw [black] (9.8,-6.9) circle (2pt);
    \filldraw [black] (9.8,-8.5) circle (2pt);
    \filldraw [black] (11.4,-6.9) circle (2pt);
    \filldraw [black] (11.4,-8.5) circle (2pt);
    \draw[thick] (9.8,-6.9) -- (9.8,-8.5) -- (11.4,-8.5) -- (11.4,-6.9) -- (9.8,-6.9);

\end{tikzpicture}
\caption{Multi-graphs in $Hom(P_4)$ and $Hom(P_5)$ corresponding to the walks of length 3 and 4.}\label{Fig:path-homomorphic-images}
\end{center}
\end{figure}
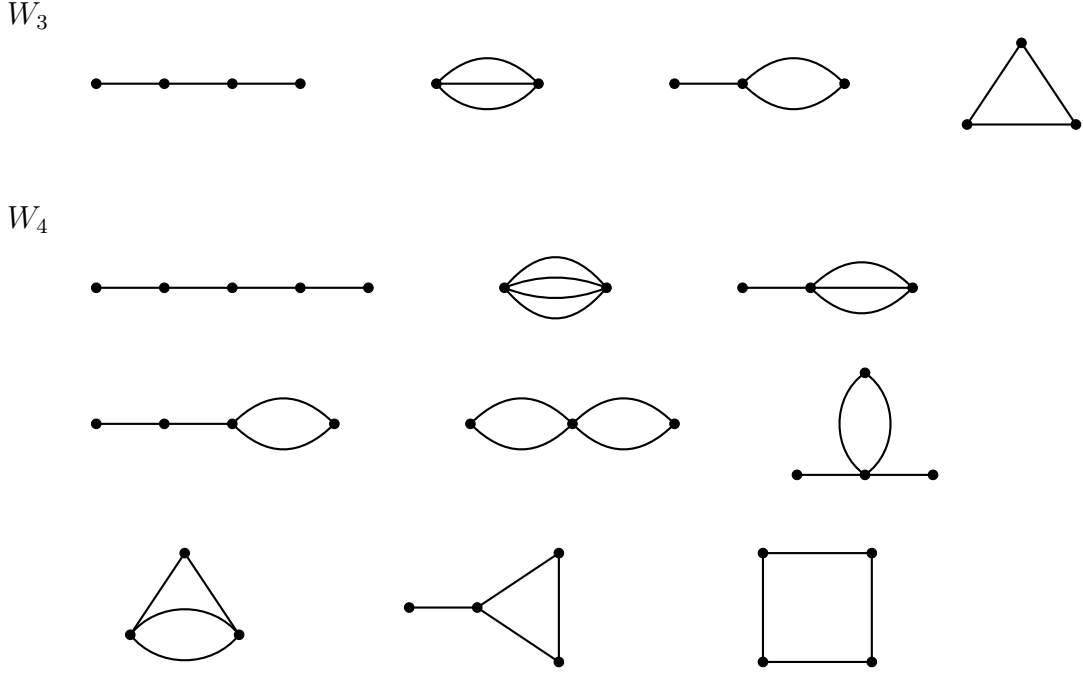

Since $ex^{\sm}_k(n,Hom(F))\le ex^{\sm}_k(n,F)$, it is relatively easier to give upper bounds on the former. As a striking example, we now present   one line proofs of the analogs of  Theorems~\ref{thm:p4sum} and~\ref{thm:p5sum}, which state that for all $n\ge1$ and $k\ge3$, $ex_k^{\sm}(n,Hom(P_4))=2\binom n2$, while for all $n\ge1$ and $k\ge4$, $ex_k^{\sm}(n,Hom(P_5))=3\binom n2$. The lower bounds follow from~\eqref{clique_bound}. For the first upper bound, suppose that $\sum_{i=1}^k e(G_i)>2\binom n2$. Then, by the pigeon-hole principle, there is a~pair of vertices $u,v$ connected by at least 3 edges, so it spans a rainbow walk $uvuv$ belonging to $Hom(P_4)$. The proof for $P_5$ is similar.

However, we will truly acknowledge the benefit of working with $Hom(F)$ rather than $F$ alone, when proving Theorem~\ref{thm:p5min}. This simplified approach has been furnished by \cite[Theorem~4.4]{IKLS25} (with injective $\psi:E(F)\to[k]$)  implying that

\begin{equation}\label{pipi}
    \pi^{\min}_k(F)=\pi^{\min}_k(Hom(F)).
\end{equation}

\subsection{Organization of the paper}

In Section~\ref{Max_sum} we prove Theorems~\ref{thm:p4sum} and~\ref{thm:p5sum}, where the sum of the numbers of edges is maximized, while in Section~\ref{Max_min} we deal with maximizing the minimum number of edges over all graphs and show Theorem~\ref{thm:p5min}. 
Section~\ref{complete} is devoted to the proof of Theorem~\ref{thm:P5_complete} considering the most interesting case of $4$ colors, but with an additional assumption that the colored multi-graph is complete. 
The final Section~\ref{fin_rem} presents our conjecture about the value of $\pi_k^{\min}(P_5)$ for all $k$.

\section{Maximizing the sum}\label{Max_sum}

In this section, we prove Theorems~\ref{thm:p4sum} and~\ref{thm:p5sum}. 
Since the lower bounds follow from~\eqref{clique_bound} and~\eqref{Turan_bound}, we focus exclusively on the upper bounds and, for that reason, we slightly restate both statements. In the proofs, we will be frequently relying (often silently) on two observations which we call, respectively, the \emph{1-2-3 rule} and the \emph{1-2-3-4 rule}.
The first of them says that any path of length three in $M$ with edge-multiplicities  at least 1, at least 2, and at least 3 (in any order), contains a rainbow $P_4$; the second rule is analogous and implies the existence of a~rainbow~$P_5$. 

\subsection{Proof of Theorem~\ref{thm:p4sum}}

In view of the lower bounds presented in Section~\ref{Isum}, it suffices to prove the following statement.

\begin{theorem}\label{thm:p4sum1}
    For all integers $n$ and $k\geq3$, every collection of graphs $G_1, \ldots, G_k$ on a common set of $n$ vertices and containing no rainbow $P_4$, satisfies
    \begin{equation}\label{sum_bound}\sum_{i=1}^k e(G_i) \leq
        \begin{cases}
        n^2-n & \mathrm{if}\quad k \leq n-1\,,\\
        kn & \mathrm{if}\quad k\geq n-1\,.\end{cases}
    \end{equation}
\end{theorem}

\begin{proof} We proceed by induction on $n\ge1$.
    For $n\in\{1,2,3\}$ the statement clearly holds (for all $k\ge3$). Let $n \geq 4$, $k\ge3$, and  $M:=M(G_1,\dots,G_k)$ be the colored multi-graph on $n$ vertices. Suppose $M$ is rainbow $P_4$-free but, to the contrary, $e(M)$ exceeds the upper bound~\eqref{sum_bound}.
    We begin by setting a lower bound on the minimum degree $\delta(M)$.
    
    \begin{claim}\label{C1}
       We have $\delta(M)\geq \max\{n,k+1\}$.
    \end{claim}
    \begin{proof}
        To the contrary, assume that there exists a vertex $v$ with $d(v)\leq \max\{n,k+1\}-1$.
        If $k\leq n-2$, then $\max\{n,k+1\}-1 = n-1$ and, by the induction assumption applied to $M-v$,
        \[e(M)\leq (n-1)^2-(n-1)+n-1=(n-1)^2 \leq n^2-n,\]
        a contradiction to the assumption on $e(M)$.
        If $k \geq n-1$, then $\max\{n,k+1\}-1 = k$ and, similarly
        \[e(M)\leq k(n-1)+k= kn,\]
        again, a contradiction.
    \end{proof}

Since $kn \geq n^2-n$ for $k \geq n-1$, we have $e(M) > n^2-n = 2\binom{n}{2}$ regardless of $k$, which, by the pigeon-hole principle, implies that $M$ contains two vertices $u,v$ with $m(u,v)\geq 3$. We fix such $u$ and $v$ and prove another claim.

    \begin{claim}\label{C2}
        There exists a vertex $w\not\in\{u,v\}$ such that all neighbors of vertices $u$, $v$ and $w$ are in the set $\{u,v,w\}$.
    \end{claim}
    \begin{proof}
        Since $d(u) \ge k+1$ and $d(v) \ge k+1$ by Claim~\ref{C1}, $N(u) \setminus \{v\} \neq \emptyset$ and $N(v) \setminus \{u\} \neq \emptyset$.
    
        Assume first that there exist two distinct vertices $x$ and $y$, other than $u$ and $v$, such that $m(u,x) \geq 1$ and $m(v,y) \ge 1$. By the 1-2-3 rule, to avoid a rainbow path $xuvy$, we must have $m(u,x) = m(v,y) = 1$. Since, by Claim~\ref{C1}, $\delta(M)\geq n$, we infer that the vertex $x$ has a~neighbor $z$ such that $m(x,z) \geq 2$. To avoid a rainbow path $zxuv$, again by the 1-2-3 rule, we must have $z=v$, so $m(x,v) \geq 2$. By a symmetric argument, $m(y,u) \geq 2$, which creates a rainbow path $yuvx$.

        Therefore, there exists a vertex $w\in N(u)\cap N(v)$ and $N(u)\cup N(v) = \{u,v,w\}$. If $m(u,w) = m(v,w) = 1$, then, since by Claim~\ref{C1} we have $d(w)\geq n$, vertex $w$ has a neighbor $t \not \in \{u,v\}$ with $m(w,t) \geq 2$, resulting in a rainbow path $uvwt$. Thus, $m(u,w) \geq 2$ or $m(v,w) \geq 2$, either of which implies that $w$ does not have neighbors different from $u$ and $v$, since otherwise a rainbow path $P_4$ would be created. This completes the proof of the claim.
    \end{proof}

We may now finish the proof of Theorem~\ref{thm:p4sum1}. Let $u,v,w$ be as in Claim~\ref{C2}, thus, they form a connected component of $M$.
By removing  $u,v$ and $w$ from $M$ at most $3k$ edges are lost, so by the induction assumption applied to $M-\{u,v,w\}$,
if $k\leq n-4$, then
        \[e(M)\leq (n-3)^2-(n-3)+3k=n^2-n-6n+3k+12\leq n^2-n,\]
while if $k \geq n-3$, then
        \[e(M)\leq k(n-3)+3k=kn.\]
Note that for $k \leq n-1$ (and, in particular, for $k\in\{n-3,n-2,n-1\}$) we have $kn \leq n^2-n$, so, in each case the obtained bound on $e(M)$ contradicts the assumption on $M$.
\end{proof}

\subsection{Proof of Theorem~\ref{thm:p5sum}}

The proof of Theorem~\ref{thm:p5sum} is based on the same idea as that of Theorem~\ref{thm:p4sum1}, but, understandably, is much more involved. In view of the existing lower bounds, it suffices to prove the following.

\begin{theorem}\label{thm:p5sum1}
    For all integers $n$ and $k\geq4$, every collection of graphs $G_1, \ldots, G_k$ on a common set of $n$ vertices and containing no rainbow $P_5$, satisfies
    \begin{equation}\label{sum_bound2}\sum_{i=1}^k e(G_i) \leq
    \begin{cases}
    \frac{3}{2}(n^2-n) & \mathrm{if}\quad k \leq n-1\,,\\
    \frac{3}{2}kn & \mathrm{if}\quad k\geq n-1\,.\end{cases}
    \end{equation}
\end{theorem}

\begin{proof}
    We proceed by induction on $n$. For $n \leq 4$ the statement clearly holds (for all $k\geq 4$). Let $n\ge5$, $k\ge4$, and  $M:=M(G_1,\dots,G_k)$ be the colored multi-graph on $n$ vertices. Suppose $M$ is rainbow $P_5$-free but, to the contrary, $e(M)$ exceeds the upper bound~\eqref{sum_bound2}.
    As before, we begin with proving a lower bound on $\delta(M)$.

    \begin{claim}\label{cla:p5sum_degree}
        We have $\delta(M) \geq n + 2$. It follows that for every vertex $x$ there is a vertex $y$ such that $m(x,y)\ge2$.
    \end{claim}
    \begin{proof}
        To the contrary, assume that there exists a vertex $v$ with $d(v) \leq n+1$. If $k\le n-2$, then, by the induction assumption applied to $M-v$,
        \[e(M)\leq \frac{3}{2}(n-1)^2-\frac{3}{2}(n-1)+n+1=\frac{3}{2}n^2-\frac{3}{2}n -2n+4 \leq \frac{3}{2}(n^2-n).\]
        If, on the other hand, $k \geq n-1$, then
        \[e(M)\leq \frac{3}{2}k(n-1)+n+1 \leq \frac{3}{2}kn -\frac{1}{2}k + 2 \leq \frac{3}{2}kn,\]
        since $k \geq 4$. This contradicts the assumption on the number of edges of $M$.
    \end{proof}

    We also need a similar claim to the one above, but allowing a removal of more vertices from $M$.

    \begin{claim}\label{cla:p5sum_set}
        Every set $S$ of $s$ vertices satisfies $e(M)-e(M-S) \geq \left\lceil\frac{3}{2}sk+\frac{1}{2}\right\rceil$. In particular,
        \begin{equation}\label{deltaM}
            \delta(M) \geq \left\lceil\frac{3}{2}k+\frac{1}{2}\right\rceil \geq k+3.
        \end{equation}
    \end{claim}
    \begin{proof}
        Assuming to the contrary that $e(M)-e(M-S) \leq \frac{3}{2}sk$ and using the induction assumption for $M-S$, one gets that if $k \leq n-s-1$, then
        \[e(M) \leq \frac{3}{2}(n-s)^2 - \frac{3}{2}(n-s) + \frac{3}{2}sk =
        \frac{3}{2}n^2-\frac{3}{2}n -\frac{3}{2}s(2n-s-k-1) \leq  \frac{3}{2}n^2-\frac{3}{2}n,\]
        while if $k \geq n-s$, then
        \[e(M) \leq \frac{3}{2}k(n-s) + \frac{3}{2}sk = \frac{3}{2}kn.\]
        Since $\frac{3}{2}kn \leq \frac{3}{2}(n^2-n)$ for $k \leq n-1$, we again obtain a contradiction with the assumption on $e(M)$.
    \end{proof}

    In the next three claims, some structures are excluded from appearing in $M$. Ultimately, in the main proof only the last one, Claim~\ref{cla:p5sum_no43}, is used, while the first two help to prove it. Still, we believe that the next claim lies at the heart of the proof of Theorem~\ref{thm:p5sum1}.

    \begin{claim}\label{cla:p5sum_star}
        There is no vertex $u$ in $M$ whose neighborhood $N(u)$ contains four  vertices $v_1$, $v_2$, $v_3$, $v_4$, such that $m(u,v_1) \geq 3$ and $m(u,v_i) \geq 4$, for $i\in\{2,3,4\}$.
    \end{claim}
    \begin{proof}
        Suppose, to the contrary, that such a vertex $u$ exists, and let $A$ be the largest subset of $N(u)$ such that for all but one vertex $v\in A$ we have $m(u,v)\ge4$, while for the remaining vertex in $A$, call it $w$, $m(u,w)\ge3$. By assumption, $a:=|A|\ge 4$.

        We first show that for every $x\in A$ and $y\neq u$, $m(x,y)\le 1$.
        To see this, assume that  $m(x,y) \geq 2$. It follows, by the 1-2-3-4 rule, that for every  vertex $z \in A\setminus\{x,y\}$, we must have $N(z)=\{u\}$. Indeed, if $m(z,t)\ge1$ for some $z\in A\setminus\{x,y\}$ and $t\not\in\{u,x,y\}$, then the 4-path $tzuxy$ would have multiplicities at least $1,3,4,2$ or $1,4,3,2$ (either $m(z,u)\ge 4$ or $m(u,x)\ge4$ and both are at least 3), leading to a rainbow $P_5$.
        But also, if either $m(x,z) \geq 1$ or $m(y,z) \geq 1$, then we have a rainbow $P_5$ ($yxzut$ or $zyxut$), using a vertex $t\in A$ different from $x,y,z$ (note that $t$ exists, since $|A|\ge4$). Thus, $N(z)=\{u\}$, so  $d(z)\le k$, which contradicts~\eqref{deltaM}.

        Next, we claim that for every $x\in A$,  $N(x)\subseteq A\cup\{u\}$. Indeed, let $y\in N(x)\setminus\{u\}$.
        Then, for every $t\neq u$ we must have $m(y,t) \leq 1$, because otherwise, using a fresh vertex $z\in A \setminus\{x,y,t\}$, the path $zuxyt$ would contain a rainbow $P_5$ (the 1-2-3-4 rule, again). Thus, by Claim~\ref{cla:p5sum_degree} applied to vertex $y$, and recalling that also $m(x,y)\le1$, one gets $m(y,u) \geq d(y)-(n-2)\ge 4$, that is, $y\in A$, by the definition of $A$.

        In summary,  there are at most $ak$ edges between $A$ and $\{u\}$, at most $\binom{a}{2}$ edges inside~$A$, and no edges between $A$ and $V(M)\setminus(A \cup \{u\})$. Therefore, by the induction assumption applied to $M-A$, if $k \leq n-a-1$, then
        \[e(M) \leq \frac{3}{2}(n-a)^2 - \frac{3}{2}(n-a) + ak + \binom a2 = \frac{3}{2}n^2-\frac{3}{2}n -a(3n-2a-k-1) \leq \frac{3}{2}n^2-\frac{3}{2}n,\]
        because $3n-2a-k-1\ge 2n-a>n>0$.
        If, on the other hand, $n-a\le k\le n-1$, then similarly,
        \begin{align*}
            e(M) &\leq \frac{3}{2}k(n-a) + ak +\binom a2 \leq \frac{3}{2}(n-1)(n-a)+a(n-1) + \frac{1}{2}a^2 - \frac{1}{2}a \\
            &= \frac{3}{2}n^2-\frac{3}{2}n - \frac{1}{2}a(n-a) \leq \frac{3}{2}n^2-\frac{3}{2}n.
        \end{align*}
        Finally,  if $k \geq n$, then
        \[e(M) \leq \frac{3}{2}k(n-a) + ak + \binom a2= \frac{3}{2}kn -\frac{1}{2}a(k-a+1) \leq \frac{3}{2}kn,\]
        because $k \geq n \geq a+1$. Thus, in all cases, we arrive at a contradiction with the assumption on $e(M)$.
    \end{proof}

    \begin{claim}\label{cla:p5sum_path}
        If three vertices $u,v,w$ satisfy $m(u,v) \geq 3$ and $m(u,w) \geq 4$, then $v$ and $w$ do not have different neighbors outside of the set $\{u,v,w\}$.
    \end{claim}
    \begin{proof}
        Suppose to the contrary that there are different vertices $x$ and $y$ not belonging to the set $\{u,v,w\}$ and such that $m(v,x) \geq 1$ and $m(w,y) \geq 1$. By the 1-2-3-4 rule, to avoid a rainbow $P_5$, we must have $m(v,x) = m(w,y) = 1$. By the same token, $m(x,y)\le1$, as well as $m(x,z)\le1$ and $m(y,z)\le 1$ for all $z\not\in\{u,v,w,x,y\}$.

        Suppose now that $m(x,w)\ge2$.
        It follows that $m(v,y)=0$, since otherwise there would be a rainbow path $yvuwx$. Thus, $\sum_{t\neq u}m(y,t)\le n-3$, and so, by Claim~\ref{cla:p5sum_degree},  $m(y,u) \geq 4$ (in fact, $m(y,u) \geq 5$), and there is a rainbow path $yuwxv$, a contradiction. Therefore, $m(x,w) \leq 1$, and so $\sum_{t\neq u}m(x,t)\le n-2$.  Hence, by  another application of Claim~\ref{cla:p5sum_degree}, $m(x,u) \geq 4$.

        By a symmetric argument we also have $m(y,u) \geq 4$. However, this contradicts Claim~\ref{cla:p5sum_star} (with $v_1=v$, and $\{v_2,v_3,v_4\}=\{w,x,y\}$).
    \end{proof}

    \begin{claim}\label{cla:p5sum_no43}
        There are no vertices $u,v,w$ such that $m(u,v) \geq 3$ and $m(u,w) \geq 4$.
    \end{claim}
    \begin{proof}
        Suppose to the contrary that there are such vertices and assume first that $v$ has at least two neighbors different from $u$ and $w$. Then, by Claim~\ref{cla:p5sum_path}, $N(w)\subseteq\{u,v\}$, so, by~\eqref{deltaM}, $d(w)=m(w,u)+m(w,v) \geq k+3$, which implies $m(w,v) \geq 3$. Thus, swapping $u$ and $w$ around, we again infer by Claim~\ref{cla:p5sum_path} that $N(u)\subseteq\{v,w\}$. This yields a contradiction with Claim~\ref{cla:p5sum_set} for $S=\{u,w\}$. Indeed,
        $$e(M)-e(M-S)=m(u,w)+m(u,v)+m(w,v)\le 3k.$$ 
        Thus, $v$ may have at most one neighbor outside $\{u,w\}$. An analogous argument shows that $w$ also has at most one neighbor outside $\{u,v\}$.

        By Claim~\ref{cla:p5sum_set}, for $S=\{v,w\}$, $v$ or $w$ must have a neighbor $x$ different from $u,v,w$, say $x\in N(v)$ (the case $x\in N(w)$ is analogous). By the conclusion of the previous paragraph and by Claim~\ref{cla:p5sum_path}, $N(v)\cup N(w)\subseteq\{u,v,w,x\}$. If vertex $x$ has a neighbor $z$ other than $u$, $v$ and $w$, then, to avoid a rainbow $P_5$, $m(x,y) \leq 1$ for all $y\neq u$, including $y=v$ and $y=w$ (1-2-3-4 rule again, see Fig.~\ref{F3_1}). By Claim~\ref{cla:p5sum_degree}, this implies that $m(x,u) \geq 4$, and, by Claim~\ref{cla:p5sum_path} applied to the triple $u,v,x$, we have $m(v,w)=0$. Thus, $d(w)\le k+1$, a~contradiction with~\eqref{deltaM}. Therefore, one can assume that $N(x)\subseteq\{u,v,w,x\}$.

\begin{figure}[ht]
\begin{center}
\begin{tikzpicture}[scale=1]

    \draw[thick] (4,0) -- (0,0) -- (2,2) -- (4,0) -- (6,0) (0,0) -- (2,-2) (4,0) -- (2,-2);
    \draw[thick, dashed] (2,2) -- (2,-2);

    \node at (0,0) [below] {$u$};
    \node at (2,2) [above] {$v$};
    \node at (4,0) [below] {$x$};
    \node at (2,-2) [below] {$w$};
    \node at (6,0) [below] {$z$};
    
    \node at (0.7,0.9) [above] {\footnotesize $\geq3$};
    \node at (0.7,-1) [below] {\footnotesize $\geq4$};
    \node at (3.3,0.9) [above] {\footnotesize $\leq1$};
    \node at (3.3,-1) [below] {\footnotesize $\leq1$};
    \node at (1.2,0) [above] {\footnotesize $\geq4$};
    \node at (5,0) [above] {\footnotesize $\geq1$};

    \draw[fill=black] (0,0) circle [radius=2pt];
    \draw[fill=black] (2,2) circle [radius=2pt];
    \draw[fill=black] (4,0) circle [radius=2pt];
    \draw[fill=black] (2,-2) circle [radius=2pt];
    \draw[fill=black] (6,0) circle [radius=2pt];

\end{tikzpicture}
\caption{A neighbor of $x$ outside $u,v,w$ leads to a contradiction on $d(w)$ in Claim~\ref{cla:p5sum_no43}.}
\label{F3_1}
\end{center}
\end{figure}
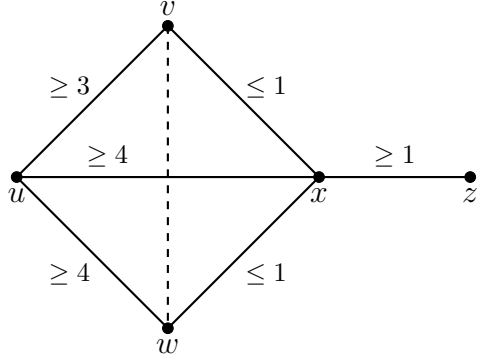 

        To avoid a contradiction with Claim~\ref{cla:p5sum_set} for $S = \{u,v,w,x\}$, vertex $u$ must have a neighbor $y$ outside of the set $\{u,v,w,x\}$ (see Fig.~\ref{F3}).
        We now show that $m(w,x)\ge1$. Indeed, if $m(w,x)=0$, then $d(w)=m(w,u) + m(w,v) \geq k+3$ by~\eqref{deltaM}, which gives a contradiction to Claim~\ref{cla:p5sum_path} for the triple $w,v,u$ (with $x$ and $y$ the witnesses). By three more ``\`a rebours'' applications of Claim~\ref{cla:p5sum_path}, we infer that $m(v,w) \leq 2$, $m(v,x) \leq 3$, and $m(w,x) \leq 2$. Moreover, $m(v,x) = 3$ and $m(w,x) = 2$ cannot hold simultaneously, since then one would have a~rainbow path $vxwuy$. Therefore, $m(v,w) + m(v,x) + m(w,x) \leq 6$. This, however, gives a~contradiction with Claim~\ref{cla:p5sum_set} for $S=\{v,w,x\}$, since $3k+6 \leq \frac{9}{2}k < \left\lceil\frac{9}{2}k+\frac{1}{2}\right\rceil$.
    \end{proof}

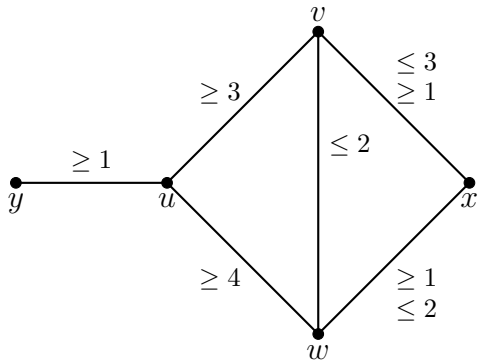
\begin{figure}[ht]
\begin{center}
\begin{tikzpicture}[scale=1]

    \draw[thick] (0,0) -- (2,2) -- (4,0) (-2,0) -- (0,0) -- (2,-2) (4,0) -- (2,-2) -- (2,2);

    \node at (0,0) [below] {$u$};
    \node at (2,2) [above] {$v$};
    \node at (4,0) [below] {$x$};
    \node at (2,-2) [below] {$w$};
    \node at (-2,0) [below] {$y$};
    
    \node at (0.7,0.9) [above] {\footnotesize $\geq3$};
    \node at (0.7,-1) [below] {\footnotesize $\geq4$};
    \node at (3.3,0.9) [above] {\footnotesize $\geq1$};
    \node at (3.3,1.3) [above] {\footnotesize $\leq3$};
    \node at (3.3,-1) [below] {\footnotesize $\geq1$};
    \node at (3.3,-1.4) [below] {\footnotesize $\leq2$};
    \node at (2,0.5) [right] {\footnotesize $\leq2$};
    \node at (-1,0) [above] {\footnotesize $\geq1$};
    
    \draw[fill=black] (0,0) circle [radius=2pt];
    \draw[fill=black] (2,2) circle [radius=2pt];
    \draw[fill=black] (4,0) circle [radius=2pt];
    \draw[fill=black] (2,-2) circle [radius=2pt];
    \draw[fill=black] (-2,0) circle [radius=2pt];

\end{tikzpicture}
\caption{Edge multiplicities yielding a rainbow path in Claim~\ref{cla:p5sum_no43}.}
\label{F3}
\end{center}
\end{figure} 

    We are ready to complete the proof of Theorem~\ref{thm:p5sum1}. We say that an edge $uv$ in $M$ is \emph{thick} if $m(u,v) \geq 4$. By Claim~\ref{cla:p5sum_no43}, thick edges are vertex disjoint. What is more, if $uv$ is a thick edge, then $m(u,w) \leq 2$ for every $w \not= v$. This implies that $M$ has the following structure: $V(M)=T\cup U$, where $T$ is the set of all endpoints of thick edges, all non-thick edges incident to vertices in $T$ have multiplicities at most 2, while all other edges, that is, those with both endpoints in $U$, have multiplicities at most 3.

    Let us bound $e(M)$ in terms of the number $t=|T|/2$  of thick edges in $M$. Note that $2t\le n$ and recall that the multiplicity of each edge in $M$ is at most $k$. We have
    \begin{align*}
        e(M) &\leq 2\cdot \binom{2t}{2} + (k-2)\cdot t + 2\cdot 2t(n-2t) + 3\cdot \binom{n-2t}{2} \\
        &= \frac{3}{2}(n^2-n) + t(2t+k-2n-1)\\
        &\leq \frac{3}{2}(n^2-n) + t(k-n-1).
    \end{align*}

    \noindent If $k \leq n-1$, then $k-n-1 < 0$ and $e(M) \leq \frac{3}{2}(n^2-n)$, contradicting the assumption on~$M$.
    If, on the other hand, $k \geq n-1$, then
    \begin{align*}
        e(M) & \le \frac{3}{2}(n^2-n) + t(k-n+1) -2t \le \frac32(n^2-n)+\frac{n}2(k-n+1) \\
        & = n\left(n-1+\frac{k}2\right)\leq \frac32 kn, 
    \end{align*}
     a contradiction again.
\end{proof}

\section{Maximizing the minimum}\label{Max_min}

Here we prove Theorem~\ref{thm:p5min}, first for $k=6$, then for $k=5$, and finally for $k=9$.
This particular order is justified by using the result for  $k=5$  in the proof for $k=9$.
Of course, the case $k=6$ follows trivially from  that of $k=5$. Nevertheless, we present a short, independent proof of the former in order to introduce a proof technique to be used also in the latter.

In view of the lower bound~\eqref{kell2}, in all cases it suffices to prove an upper bound, which will be formulated each time as a separate statement. Moreover, based on identity~\eqref{pipi}, it suffices to prove it for $Hom(P_5)$, that is, the family of all walks $W_4$ of length four, instead of just for~$P_5$. In particular, for any multi-graph $M$ not containing a rainbow $W_4$, we will be assuming that
\begin{equation}\label{mle3}
\mbox{for every edge $e$ in $M$, we have $m(e)\le3$,}
\end{equation}
since a quadruple edge itself forms a rainbow $W_4$.

For any vertex $v$ of $M$, set 
$$t(v)=|\{u: m(v,u)=3\}|,\quad p(v)=|\{u: m(v,u)=2\}|\quad\mbox{and}\quad s(v)=|\{u: m(v,u)=1\}|,$$ and notice that, owing to~\eqref{mle3},
\begin{equation}\label{tps}
|N(v)|=t(v)+p(v)+s(v)\quad\mbox{and}\quad d(v)=3t(v)+2p(v)+s(v).
\end{equation}

We begin with a basic observation on $Spe(v)=\{i\in[k]: v\in \overline{V}(G_i)\}$. Recall that $spe(v)=|Spe(v)|$.
\begin{observation}\label{obs:double_3colors}
    For every $k\ge4$, $M=M(G_1, \ldots, G_k)$ with no  rainbow $W_4$, and $v\in V(M)$, if $p(v)+t(v)\ge1$, then $spe(v)\le3$. 
\end{observation}
\begin{proof} 
The assumption $p(v)+t(v)\ge1$ implies that there exists a vertex $u$ satisfying \hbox{$uv \in E(G_i) \cap E(G_j)$} for $i \not=j$.
If $spe(v) \ge 4$, then there exist $x$ and $y$ such that $vx \in E(G_h)$ for $h \not\in\{i,j\}$ and $vy \in E(G_\ell)$ for $\ell \not\in\{i,j,h\}$. This gives the forbidden rainbow walk $xvuvy$. 
\end{proof}

We now introduce a crucial notion of the base graph.
For $k\ge3$, a colored multi-graph $M(G_1, \ldots, G_k)$ on vertex set $V$, and $1\le i <j<h\le k$, let $\bg = \bg(i, j, h)$ be the graph on $V$ in which $uv$ is an edge if and only if there exists a rainbow path of length two from $u$ to $v$ in $M(G_i, G_j, G_h)$. We will call  $\bg$ the \emph{base graph} of $G_i, G_j, G_h$.

The next result reveals the role of base graphs in our proofs.

\begin{claim}\label{cla:base}
For all $\eps, \eps'>0$ there exists $n_0$ such that for all $n\ge n_0$ and every  collection of graphs $G_1, G_2, \ldots, G_k$, $k \geq 5$, on a common set $V$ of $n$ vertices satisfying $e(G_i)\ge\left(\frac{1}{8}+\eps\right)n^2$ for $i \in [k]$,
if there are indices $1\le i <j<h\le k$ such that the base graph  $\bg = \bg(i, j, h)$ satisfies $e(\bg)\ge \left(\frac{1}{8}+\eps'\right) n^2$, 
then the multi-graph $M(G_1, G_2, \ldots, G_k)$ contains a rainbow $W_4$.
\end{claim}
\begin{proof}
W.l.o.g.~assume that $\{i,j,h\}=\{1,2,3\}$. By Theorem~\ref{thm:p4min} with $k=3$, for any $\epsilon, \epsilon'>0$ there exists $n_0$ such that any multi-graph $M(\bg, G_4, \ldots, G_k)$ on $n \geq n_0$ vertices satisfying $e(\bg)\ge \left(\frac{1}{8}+\eps'\right) n^2$ and $e(G_i)\ge\left(\frac{1}{8}+\eps\right)n^2$ for $4 \leq i \leq k$, contains a rainbow walk of length three. Upon replacing the edge of $\bg$ in that walk with a rainbow walk of length two in colors from $\{1,2,3\}$, we obtain a rainbow $W_4$ in $M(G_1, G_2, \ldots, G_k)$.	
\end{proof}

Below are a couple of simple counting facts useful for estimating the size of the base graph in various circumstances.
\begin{claim}\label{Poznanski_lemacik}
Let $M(G_1, \ldots, G_k)$ be a colored multi-graph on $n$ vertices and $v \in V(M)$.
\begin{itemize}
\item[(i)]
For any real $x \geq \frac{1}{2}n$ and any three colors $i,j,h$,  if $d_i(v)+d_j(v)+d_h(v)\ge 3x$, then $e(\bg(i,j,h))\ge\binom{x}{2}$.

\item[(ii)]
If $Spe(v)\subseteq\{i,j,h\}$, then $e(\bg(i,j,h))\ge \binom{p(v)}2+p(v)s(v)$.
\end{itemize}
\end{claim}

\begin{proof} Set $\bg=\bg(i,j,h)$.

(i) Set $A=N_i(v)$, $B=N_j(v)$, and $C=N_h(v)$. It is easy to see that $\bg$ contains all pairs $\{u,w\}$ of distinct vertices such that
$(u,w) \in (A\times B) \cup (B\times C) \cup (A\times C)$.
W.l.o.g.~assume that $|A|\geq |B|\geq |C|$, set $a = |A|$ and note that $|B| \geq \frac{3x-a}{2}$.
The desired number of pairs $\{u,w\}$ is thus at least $\frac12 (|A|-1)\cdot|B|$, because for each $w \in B$ we can take any $u \in A$ different from $w$, and we need to divide by $2$ to avoid counting twice pairs inside $A \cap B$.
Since $\frac12 (|A|-1)\cdot|B| \geq \frac{(a-1)(3x-a)}{4}$ and $x \leq a \leq n \leq 2x$, the last quantity is at least $\binom{x}{2}$.

(ii) Let $P(v)=\{u: m(u,v)=2\}$ and $S(v)=\{u: m(u,v)=1\}$. It is easy to see that  every pair of vertices $\{u,w\}$ such that both $u$ and $w$ are in $P(v)$, or one is in $P(v)$ and the other in $S(v)$, forms a length two rainbow path in colors from $\{i,j,h\}$ and with the middle vertex $v$. Then each such pair is an edge in $\bg$ and the assertion follows.
\end{proof}

\subsection{Six colors}
We are to prove the following statement.

\begin{theorem}\label{thm:p5min_upper6}
For every $\eps>0$ there exists $n_0$ such that for all $n\ge n_0$ and every  collection of graphs $G_1, G_2, \ldots, G_6$ on a common set $V$ of $n$ vertices, each with  $e(G_i)\ge\left(\frac{1}{8}+\eps\right)n^2$, the multi-graph $M(G_1, G_2, \ldots, G_6)$ contains a rainbow $W_4$.
\end{theorem}

\begin{proof}
Suppose by contradiction that $G_1, G_2, \ldots, G_6$ is a collection of graphs on $V$, each with at least $\left(\frac{1}{8}+\eps\right)n^2$ edges, not containing a rainbow $W_4$.
Since $\sum_i e(G_i) \geq \left(\frac{3}{4}+6\eps\right)n^2$, there exists a vertex $v$ with $d(v) \geq \left(\frac{3}{2}+12\eps\right)n$. It follows that for some vertex $u$ we have \hbox{$m(v,u)\ge2$}, say $vu\in E(G_1)\cap E(G_2)$. Since there is no rainbow $W_4$, by Observation~\ref{obs:double_3colors}, vertex $v$ may have neighbors in at most one other graph. W.l.o.g.~we may assume that $Spe(v)\subseteq\{1,2,3\}$.
Consider the base graph $\bg=\bg(1,2,3)$.
Applying Claim~\ref{Poznanski_lemacik}\,(i) with $x=\left(\frac{1}{2}+4\eps\right)n$, we infer that $e(\bg)\ge\binom{x}{2}\ge\left(\frac{1}{8}+\eps\right)n^2$ for large enough $n$. Thus, by Claim~\ref{cla:base}, $M$~contains a~rainbow walk $W_4$, a contradiction.
\end{proof}

\subsection{Five colors}\label{5colors}

We are to prove the following statement.

\begin{theorem}\label{thm:p5min_upper5}
For every $\eps>0$ there exists $n_0$ such that for all $n\ge n_0$ and every  collection of graphs $G_1, G_2, \ldots, G_5$ on a common set $V$ of $n$ vertices, each with  $e(G_i)\ge\left(\frac{1}{8}+\eps\right)n^2$, the multi-graph $M(G_1, G_2, \ldots, G_5)$ contains a rainbow $W_4$.
\end{theorem}

\begin{proof}

Suppose to the contrary that for some $\eps>0$ and arbitrarily large $n$ there is a collection of graphs $G_1, G_2, \ldots, G_5$ on $V$, each with at least $\left(\frac{1}{8}+\eps\right)n^2$ edges, and such that the multi-graph $M:=M(G_1, G_2, \ldots, G_5)$ contains no rainbow $W_4$. Assume further that $M$ is maximal with these properties, that is, adding to $M$ a new edge (of any color) creates a rainbow $W_4$. Clearly,
for each $i\in\{1,2,3,4,5\}$
$$|\overline{V}(G_i)|\ge \left(\frac12+\eps\right)n.$$

\begin{observation}\label{obs:size_of_N(v)}
Let $v \in V$ and $P(v) := \{u\in V: m(v,u)=2 \}$.
\begin{itemize}
\item[(a)] If $t(v) \geq 1$ then for every color $i\not\in Spe(v)$, it holds that $\overline{V}(G_i)\cap (N(v)\cup\{v\})=\emptyset$. In particular $|N(v)| < \left(\frac12-\eps\right)n$.
    \item[(b)] If $p(v)\geq 1$ then there exists a color $i\not\in Spe(v)$, such that $\overline{V}(G_i)\cap P(v)=\emptyset$. In particular $p(v) < \left(\frac12-\eps\right)n$.
\end{itemize}
\end{observation}

\begin{proof}For the first part, note that if $\overline{V}(G_i)\cap (N(v)\cup\{v\})\not=\emptyset$, then there exists a rainbow $W_4$. Since $e(G_i) \ge \left(\frac{1}{8}+\eps\right)n^2$, we have $|\overline{V}(G_i)| \ge \left(\frac12+\eps\right)n$. Thus, $\overline{V}(G_i)\cap (N(v)\cup\{v\})=\emptyset$ implies that $|N(v)| < \left(\frac12-\eps\right)n$.

For the second statement, by Observation~\ref{obs:double_3colors}, $spe(v) \le 3$, say $Spe(v) = \{1,2,3\}$. If $\overline{V}(G_4)\cap P(v) \not=\emptyset$ and $\overline{V}(G_5)\cap P(v) \not=\emptyset$, then there exists a rainbow $W_4$. The inequality $p(v) < \left(\frac12-\eps\right)n$ follows by the same argument as in part (a). 
\end{proof}

\begin{claim}\label{cla:n+t}
    For every vertex $v \in V$ with $t(v)\geq 1$ we have $d(v) < n+t(v)$.
\end{claim}

\begin{proof}
    Assume to the contrary that there exists a vertex $v\in V$ such that $t(v)\ge1$ and  $d(v) \geq n+t(v)$.
    Then, by Observation~\ref{obs:size_of_N(v)}\,(a), we have $|N(v)|<\frac n2$, and, by~\eqref{tps}, it holds that  $2(t(v)+p(v)+s(v))=2|N(v)|<n$. On the other hand, $2(t(v)+p(v)+s(v)) \ge d(v)-t(v)\geq n$, a~contradiction.
\end{proof}

\begin{claim}\label{cla:cliques_in_triple_edges}
    For any triple of vertices $u,v,w$, if $m(v,u)=m(v,w)=3$, then $m(u,w)=3$. 
\end{claim}

\begin{proof}
    By Observation~\ref{obs:double_3colors}, $Spe(u)=Spe(v)=Spe(w)=\{i,j,h\}$. Suppose that one of the three colors, say $i$, satisfies $e = uw \not\in E(G_i)$. Let $M+e$ denote the colored multi-graph obtained by adding $e$ to $G_i$.  By maximality of $M$, this means that after adding the edge $e$ to $G_i$, a rainbow $W_4$ would be created. If $e$ is one of the end-edges of that rainbow walk, then we can replace it with one of the edges $uv$ or $wv$ in the graph $G_i$, thus getting a rainbow $W_4$ in $M$, a contradiction.

    So, suppose $e$ is one of the inner edges of a rainbow $W_4$ in $M+e$. By symmetry, we can assume that $e$ is the second edge of a rainbow walk $xuwyz$. Clearly, this walk has an edge in color $\ell\notin\{i,j,h\}$, and that edge has to be $yz$. 
    But then one gets a rainbow $W_4$ consisting of  $yz$, $yw$, and two edges between $w$ and $v$ of colors different from the color of  $wy$, a contradiction again.
\end{proof}

\begin{claim}\label{cla:clique_1/4}
    There exists a vertex $v$ such that  $t(v)\ge\frac{n}{4}$. It follows that there are three colors $i,j,h$, and a set $K\subset V$ with $m:=|K|\ge\frac{n}{4}+1$, such that $G_i[K]=G_j[K]=G_h[K]\cong K_m$.
\end{claim}

\begin{proof}
    Since \[ \sum_{v \in V} d(v) = 2 e(M) > \frac{5}{4}n^2,\] there exists a vertex $v$ with $d(v) > \frac{5}{4}n$.

    If $t(v)=0$, then $d(v)=2p(v)+s(v)>\frac{5}{4}n$, and so $p(v)\ge \frac n8$. By Observation~\ref{obs:double_3colors}, $spe(v)\le3$, and by Observation~\ref{obs:size_of_N(v)}\,(b), $p(v)\le \left(\frac12-\eps\right)n$. Then, by Claim~\ref{Poznanski_lemacik}\,(ii), setting $p=p(v)$ and $s=s(v)$, for some colors $i,j,h$,
    \[ e(\bg(i,j,h)) \ge \frac{p(p-1)}{2} + ps > \frac{p(p-1)}{2} + \left(\frac{5}{4}n-2p\right)p=p\left(\frac54n-\frac32p -\frac12\right). \]
    
    Since $\frac{n}{8}\le p\le\frac{n}{2}$, we have $e(\bg(i,j,h))\ge\min\left\{\frac{17n^2}{8\cdot16}-\frac{n}{16}, \frac{n^2}4-\frac{n}4\right\}>\left(\frac18+\eps\right)n^2$ for sufficiently large $n$, and, by Claim~\ref{cla:base}, there is a rainbow $W_4$, a contradiction.
    
    If $t(v)>0$, then, by Claim~\ref{cla:n+t}, $\frac{5}{4}n<n+t(v)$, which implies $t(v)>\frac{n}{4}$. By Observation~\ref{obs:double_3colors}, $spe(v)=3$. By Claim~\ref{cla:cliques_in_triple_edges}, for every pair of vertices $u,w$ such that $m(v,u)=m(v,w)=3$, we also have $m(u,w)=3$. Thus, the conclusion follows with $K=\{u: m(v,u)=3\}\cup\{v\}$ and $\{i,j,h\}=Spe(v)$.
 \end{proof}

The next result extends Claim~\ref{cla:n+t} to all $v\in V$.
\begin{claim}\label{t0d<n }
    For every vertex $v \in V$ with $t(v)=0$ we have $d(v) < n$. 
\end{claim}

\begin{proof}
    Assume there exists a vertex $v \in V$ such that $t(v)=0$ and $d(v) \geq n$.
    Then $p(v)\ge1$, hence by Observation~\ref{obs:double_3colors}, $spe(v)\le3$ and, by Observation~\ref{obs:size_of_N(v)}\,(b),
    \[p(v)\leq\left(\frac12-\eps\right)n.\]
    By Claim~\ref{cla:clique_1/4}, there exists a clique $K$ in three colors, say $1, 2, 3$, with $|K|>\frac{n}{4}$. We will now show that in some three colors we can find a large base graph $\bg(i,j,h)$, getting a contradiction due to Claim~\ref{cla:base}.	
	
    Assume first that $d_4(v)>0$ or $d_5(v)>0$. Then $N(v)\cap K=\emptyset$, as otherwise there would be a rainbow $W_4$.
    It follows that $s(v)+p(v)=|N(v)|<\frac{3}{4}n$ and $n \leq s(v)+2p(v) < \frac{3}{4}n+p(v)$, so $p(v)>\frac14n$.
    Let $Spe(v)\subseteq\{i,j,h\}$. Then, by Claim~\ref{Poznanski_lemacik}\,(ii), setting $p=p(v)$ and $s=s(v)$,
	\[e(\bg(i,j,h)) \geq \frac{p(p-1)}{2}+s \cdot p \geq \frac{p(p-1)}{2} + p(n-2p) = \frac{p}{2}(2n-3p-1). \]
	Since $\frac14n < p \leq \left(\frac{1}{2}-\eps\right)n$, we derive $e(\bg(i,j,h)) > \left(\frac{1}{8}+\frac\eps4\right)n^2$ for sufficiently large $n$, and get a contradiction by Claim~\ref{cla:base}.

    Next, assume that $Spe(v)\subseteq\{1,2,3\}$. We have $\sum_{w \in V} \left(d_4(w)+d_5(w)\right) \geq \frac{n^2}{2}$ and,
	since $\overline{V}(G_4\cup G_5)\cap K=\emptyset$, also $|\overline{V}(G_4\cup G_5)|\le\frac{3}{4}n$. Thus, by averaging, 	
	there exists a vertex $w$ such that $d_4(w)+d_5(w) \geq \frac{n^2/2}{(3/4)n}=\frac23n$.

    Now if both $d_4(w)>0$ and $d_5(w)>0$, then the number of vertices which are endpoints of a rainbow path $P_3$ in colors $4$ and $5$ is at least $\max\left\{d_4(w),d_5(w)\right\}\ge\frac12 (d_4(w)+d_5(w)) \geq \frac{n}{3}$. To avoid a rainbow $W_4$, these vertices cannot be adjacent to $v$, so $s(v)+p(v)\leq n-\frac13n = \frac23n$.
    Hence, $n \leq s(v)+2p(v) \leq \frac{2}{3}n+p(v)$, and consequently, $p(v) \geq \frac13n>\frac14 n$. Therefore, again 
    we have $e(\bg(1,2,3)) > \left(\frac{1}{8}+\frac\eps4\right)n^2$, giving a contradiction.
	
    We are left with the case $d_5(w)=0$ (or, symmetrically, $d_4(w)=0$). 
    Let us split the set of vertices $V=K\cup N_4(w)\cup R$, where the three sets are disjoint, and recall that $|K| \geq n/4$ and $|N_4(w)| \geq 2n/3$.
    We will focus on $G_5$. We have $\overline{V}(G_5)\cap K=\emptyset$ and there are at most 
    $$\binom{3n/4}2-\binom{2n/3}2<\left(\frac{9}{32}-\frac 29\right)n^2=\frac{17}{288}n^2$$
    pairs of vertices $\{u_1,u_2\}$ with $\{u_1,u_2\}\cap K=\emptyset$ and  $\{u_1,u_2\}\cap R\neq\emptyset$.
    Hence, at least $\frac18n^2 - \frac{17}{288}n^2 = \frac{19}{288}n^2$ edges of $G_5$ have both endpoints in $N_4(w)$. This implies that (much) more than $\frac14n$ vertices in $N_4(w)$ are incident to an edge of $G_5$. These vertices cannot be adjacent to $v$, so 
    
    $$|N(v)|=s(v)+p(v) < n-\frac14n = \frac34n\quad\mbox{and}\quad n \leq s(v)+2p(v) < \frac{3}{4}n+p(v),$$ yielding $p(v)>\frac14n$.
    
    Therefore, one more time, $e(\bg(1,2,3)) >\left(\frac{1}{8}+\frac\eps4\right)n^2$, leading to a final contradiction.
\end{proof}

In the remainder of the proof a pivotal role will be played by three-colored cliques. A set of vertices $K\subseteq V(M)$ is called a \emph{three-colored clique} (in colors $i,j,h$) if for all $u,v\in K$ we have $uv\in E(G_i)\cap E(G_j)\cap E(G_h)$ and no superset of $K$ has this property.

By Claim~\ref{cla:clique_1/4}, we know that there is a three-colored clique of size at least $n/4$. 
Let $K'$ and $K''$ be the two largest three-colored cliques with different sets of colors, and w.l.o.g. let $1$ be a color that occurs in both of them. If all three-colored cliques have the same set of colors, then let $1$ be any of these colors.

We now prove the following claim which will provide a bound on the total number of edges in colors different from $1$.

\begin{claim}\label{cla:5col2345}
    For any vertex $v \in V$ we have $d_2(v)+d_3(v)+d_4(v)+d_5(v) \leq n$.
\end{claim}	

\begin{proof}
    We may assume that $t(v)>0$, since otherwise the conclusion follows by Claim~\ref{t0d<n }. In turn, by Claim~\ref{cla:cliques_in_triple_edges}, $v$ belongs to a three-colored clique $K$ of size at least $t(v)+1$.
    If $1\in Spe(v)$, then $d_1(v)\ge t(v)$ and, by Claim~\ref{cla:n+t}, $d(v)-d_1(v)<n+t(v)-t(v)=n$, so the assertion follows. Hence, assume that  color 1  does not appear on $K$; say, $Spe(v)=\{3,4,5\}$. In particular, $d_1(v)=0$. Moreover, by Observation~\ref{obs:size_of_N(v)}\,(a), 
    \begin{equation}\label{eq:N(v)_disjoint}
        (N(v)\cup\{v\})\cap \overline{V}(G_i)=\emptyset,\;\; i=1,2.
    \end{equation}
    
    Let us consider two cases depending on whether  $K'$ and $K''$, in addition to 1, share color 2 or not. In either case, $\min\{|K'|,|K''|\}\ge|K|\ge t(v)+1$, since $K$ is different from $K'$ and $K''$. It follows that $t(v)<n/3$.

    \smallskip
    {\bf Case 1:} $K'$ and $K''$ do not share color 2. 
    Let color 2 be not a color of, say,  $K'$. Then, by Observation~\ref{obs:double_3colors}, $\overline{V}(G_2)\cap K'=\emptyset$ and, by \eqref{eq:N(v)_disjoint}, $N(v)\cap \overline{V}(G_2)=\emptyset$. Moreover, since $K'\subseteq \overline{V}(G_1)$, again by \eqref{eq:N(v)_disjoint}, $N(v)\cap K'=\emptyset$. Thus, all three sets, $N(v),K'$, and $\overline{V}(G_2)$ are pairwise disjoint, and, using the bounds $|\overline{V}(G_2)|\ge\frac n2$ and $|K'|\ge t(v)$, we get

    $$|N(v)|\le|V(M)|-|\overline{V}(G_2)|-|K'|\le n-\frac{n}{2}-t(v)=\frac n2-t(v),$$ and by~\eqref{tps}, dropping $(v)$ from notation,
    $$d= 3t+2p+s\le t+2(t+p+s) \leq t + 2\left(\frac{n}{2}-t\right)  = n-t<n,$$
    as needed.

    \smallskip
    {\bf Case 2:}  $K'$ and $K''$ share  colors $1$ and $2$, both missing from $Spe(v)$.  Assume that $|K'|\ge |K''|$ and set $G_{12}=M(G_1, G_2)$.
    
    Consider a vertex $u \in V(M)$. If for some $w\in K'$ and $i\in\{1,2\}$, $uw\in E(G_i)$, then $uw'\not\in E(G_{3-i})$ for all $w'\in K''$, since otherwise there would be a rainbow $W_4$. 
    By symmetry, the same is true with $K'$ and $K''$ swapped. Thus, every vertex of the set 
    $$S:=\overline{V}(G_{12})\setminus(K'\cup K'')$$
     has in $G_{12}$ at most $2|K'|$ edges going to $K'\cup K''$, 
     and there are no edges between $K'$ and $K''$ in $G_{12}$. 

    We construct an auxiliary colored multi-graph $G'_{12}$ obtained from $G_{12}$ by removing all existing edges between $S$ and $K''$ and adding all edges between $S$ and $K'$ in colors $1$ and $2$. Observe that $\overline{V}(G_{12})=\overline{V}(G'_{12})$, $e(G'_{12})\ge e(G_{12})$ and $K''$ is a connected component of $G'_{12}$. 
    
By Observation~\ref{obs:size_of_N(v)}\,(a), $N(v)\cap \overline{V}(G'_{12})=\emptyset$, so if $|\overline{V}(G'_{12})| \ge \frac{n+t(v)}{2}$, then
\[d(v)\le 3t(v)+2(n-t(v)-|\overline{V}(G'_{12})|) \le n,\]
as needed. Thus, assume that $|\overline{V}(G'_{12})| < \frac{n+t(v)}{2}$. Setting $t=t(v)$ and $a=|K''|$, we have $a < \frac{n+t}{4}$, and 
\[e(G'_{12}) \le 2\binom{\frac{n+t}{2}-a}{2} + 2\binom{a}{2} < \left(\frac{n+t}{2}-a\right)^2 + a^2.\]
The right-hand side is a decreasing function of $a$ for $a \le \frac{n+t}{4}$, and $a>t$, so
\[e(G'_{12}) < \frac{(n-t)^2}{4} + t^2 \le \max\left\{\frac{n^2}{4}, \frac{2n^2}{9}\right\}\]
since $t \in [0, \frac{n}{3}]$. 
This contradicts the assumption $e(G'_{12}) \ge 2\left(\frac{1}{8}+\eps\right)n^2$. 
\end{proof}

To finalize the proof of Theorem~\ref{thm:p5min_upper5}, note that by Claim~\ref{cla:5col2345},
\[\sum_{i=2}^5 e(G_i) = \frac{1}{2}\sum_{v\in V}\sum_{i=2}^5 d_i(v) \le \frac{1}{2}n^2.\]
On the other hand, $\sum_{i=2}^5 e(G_i) > 4\cdot \frac{1}{8}n^2$, which gives a contradiction.
\end{proof}

\subsection{Nine colors}

We are to prove the following statement.

\begin{theorem}\label{thm:p5min_upper9}
For every $\eps>0$ there exists $n_0$ such that for all $n\ge n_0$ and every  collection of graphs $G_1, G_2, \ldots, G_9$ on a common set $V$ of $n$ vertices, each with  $e(G_i)\geq\left(\frac{1}{18}+\eps\right)n^2$, the multi-graph $M(G_1, G_2, \ldots, G_9)$ contains a rainbow $W_4$.
\end{theorem}

\begin{proof}
Suppose to the contrary that for some $\eps>0$ and arbitrarily large $n$ there is a collection of graphs $G_1, G_2, \ldots, G_9$ on $V$, each with at least $\left(\frac{1}{18}+\eps\right)n^2$ edges, and such that the multi-graph $M:=M(G_1, G_2, \ldots, G_9)$ contains no rainbow $W_4$.

\begin{claim}\label{cla:9col_n/3}
    Every set $S\subset V$ of size $|S|\ge\frac{1}{3}n$  is incident to edges in at least five different graphs, that is, $\big|\bigcup_{v\in S}Spe(v)\big|\ge5$.
\end{claim}

\begin{proof} For mere convenience, assume that $n$ is divisible by 3 and suppose that a set $S$ of size $|S|=\frac13n$ is incident to edges in at most four graphs, say $G_6,G_7,G_8,G_9$.
Then $V\setminus S$ has size $\frac{2}{3}n$  and contains all edges of $G_1,\dots,G_5$. Recall that for each $i$,
$$e(G_i)\ge\left(\frac{1}{18}+\eps\right) n^2= \left(\frac{1}{8} + \eps'\right)\left(n'\right)^2,$$ where $n'=\frac{2}{3}n$ and $\eps'=\frac94\eps$. Thus, by Theorem~\ref{thm:p5min_upper5}, the multi-graph $M(G_1,\ldots,G_5)[V\setminus S]$ contains a rainbow walk on four edges, a contradiction.
\end{proof}

Since $e(M) = \sum_{i=1}^9 e(G_i) > \frac{1}{2}n^2$, there exists a vertex $v$ with $d(v)\geq n$. There are three immediate consequences of this fact (and~\eqref{tps}). Recall that every edge in $M$ has multiplicity at most $3$ (cf. \eqref{mle3}).
 \begin{observation}\label{3consec} If $d(v)\geq n$, then
  \begin{itemize}
  \item[(i)] $p(v)+t(v)\ge1$,
  \item[(ii)] $|N(v)|\geq n/3$,
  \item[(iii)] $t(v)=0$ implies that $p(v)+s(v)=|N(v)| \geq n-p(v)$.
  \end{itemize}
  \end{observation}
Let us fix a vertex $v$ with $d(v)\geq n$. A careful analysis of the neighborhood $N(v)$ will lead to the contradictory conclusion that $e(G_i)\le n^2/18$ for some $i\in\{1,2,\dots,9\}$.  
Observations~\ref{3consec}\,(i) and~\ref{obs:double_3colors}\ imply that $2\le spe(v)\le3$. W.l.o.g.~let $Spe(v)\subseteq\{1,2,3\}$ and assume  there is a vertex $u$ with $vu\in E(G_1)\cap E(G_2)$. 

Observation~\ref{3consec}\,(ii) yields, by Claim~\ref{cla:9col_n/3} with $S=N(v)$, the following observation.
\begin{observation}\label{obs:9col-neighb}
    There exist $w,w'\in N(v)$, not necessarily distinct, such that $w\in \overline{V}(G_i)$ and $w'\in \overline{V}(G_{i'})$ for some $i,i' \ge 4$, $i\neq i'$ (see Fig.~\ref{F:(*)}).
\end{observation}

\begin{figure}[ht]
\begin{center}
\begin{tikzpicture}[scale=1]
    \draw[thick] (-1,0) -- (0,1.6) -- (1,0);
    \draw[thick, color=gptblue] (-1,0) -- (-1,-1.6);
    \draw[thick, color=brilliantrose] (1,0) -- (1,-1.6);
    \draw (-2,1.6) .. controls (-1.6,2) and (-0.4,2) .. (0,1.6);
    \draw (-2,1.6) .. controls (-1.6,1.2) and (-0.4,1.2) .. (0,1.6);
    \draw[fill=black] (-2,1.6) circle [radius=2pt];
    \draw[fill=black] (1,0) circle [radius=2pt];
    \draw[fill=black] (0,1.6) circle [radius=2pt];
    \draw[fill=black] (-1,0) circle [radius=2pt];
    \draw[fill=black] (1,0) circle [radius=2pt];
    \draw[fill=black] (-1,-1.6) circle [radius=2pt];
    \draw[fill=black] (1,-1.6) circle [radius=2pt];

    \node at (-1,-0.8) [left, color=gptblue] {$i$};
    \node at (1,-0.8) [right, color=brilliantrose] {$i'$};
    \node at (-1,1.9) [above] {$1$};
    \node at (-1,1.3) [below] {$2$};
    \node at (0,1.6) [above right] {$v$}; 
    \node at (-2,1.6) [above left] {$u$};
    \node at (-1,0) [left] {$w$};
    \node at (1,0) [right] {$w'$};
    \node at (-1,-1.6) [left] {$x$};
    \node at (1,-1.6) [right] {$x'$};

    \draw[thick] (5,0) -- (5,1.6);
    \draw (3,1.6) .. controls (3.4,2) and (4.6,2) .. (5,1.6);
    \draw (3,1.6) .. controls (3.4,1.2) and (4.6,1.2) .. (5,1.6);
    \draw[thick, color=gptblue] (5,0) -- (4,-1.6);
    \draw[thick, color=brilliantrose] (5,0) -- (6,-1.6);
    \draw[fill=black] (3,1.6) circle [radius=2pt];
    \draw[fill=black] (5,1.6) circle [radius=2pt];
    \draw[fill=black] (5,0) circle [radius=2pt];
    \draw[fill=black] (4,-1.6) circle [radius=2pt];
    \draw[fill=black] (6,-1.6) circle [radius=2pt];

    \node at (4,1.9) [above] {$1$};
    \node at (4,1.3) [below] {$2$};
    \node at (5,1.6) [above right] {$v$}; 
    \node at (3,1.6) [above left] {$u$};
    \node at (5,0) [right] {$w=w'$};
    \node at (4,-1.6) [left] {$x$};
    \node at (6,-1.6) [right] {$x'$};
    \node at (4.5,-0.8) [left, color=gptblue] {$i$};
    \node at (5.5,-0.8) [right, color=brilliantrose] {$i'$};
\end{tikzpicture}
\caption{Possible scenarios for Observation~\ref{obs:9col-neighb}: $N(v)$ has at least five colors in its spectrum. In particular, $Spe(v)\subseteq\{1,2,3\}$ and $i,i'\geq 4$, $i\neq i'$. The presence and location of the edges of color 3 are not important for us.}\label{F:(*)}
\end{center}
\end{figure}
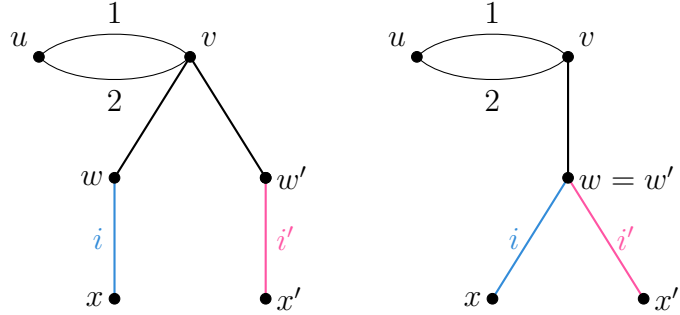

\noindent If $Spe(v)=\{1,2\}$, then there is also a vertex $x\in N(v)\cap\overline{V}(G_{j}) $, for some $j\not\in\{i,i'\}$, $j\ge 3$, but this is irrelevant for us.

Note that Observation~\ref{obs:9col-neighb} implies that $t(v)=0$ as otherwise we have a rainbow $W_4$. Hence, by Observation~\ref{3consec}\,(iii),  
\[|N(v)| \geq n-p(v).\] 

\begin{claim}\label{cla:cons_of_(*)}
There exists $j\in\{1,2\}$ such that for any vertex $w\in N(v)\cap \bigcup_{i=4}^9\overline{V}(G_i)$ we have \hbox{$m(v,w)=1$} and $vw\in E(G_j)$.
\end{claim}

\begin{proof}
Let $w\in N(v)\cap \overline{V}(G_i)$ for some $i\geq 4$ and let $x$ be such that $wx\in G_i$. Clearly, $vw\not\in G_3$, since then $xwvuv$, for any $u$ such that $vu\in E(G_1)\cap E(G_2)$, would form a rainbow $W_4$. By Observation~\ref{obs:9col-neighb}, there exists a vertex $w'\in N(v)$, a vertex $x'\in V$, and a color $i'\geq 4$, $i'\neq i$, such that $w'x'\in G_{i'}$ (see Fig.~\ref{F:(*)}). Suppose $m(v,w)\geq 2$. Then,  $xwvw'x'$ is a rainbow walk, since one can choose the color of the edge $wv$ to be different from that of the edge $vw'$. Hence $m(v,w)=1$.
Similarly, if the color of the edge $vw$ is different from the color of the edge $vw'$, then $xwvw'x'$ is a rainbow $W_4$. Thus, there is $j\in\{1,2\}$  such that $m(v,w)=1$ and $vw\in E(G_j)$ for all $w\in N(v)\cap \bigcup_{i=4}^9\overline{V}(G_i)$. 
\end{proof}

W.l.o.g.~assume that the unique color $j$ given by Claim~\ref{cla:cons_of_(*)} is 1. Let 
\[X:=\{u\in V: Spe(u)\subseteq\{1,2,3\}\}\] and 
\[P:=\{u\in V: m(v,u)=2\}.\] 
Note that by Claim~\ref{cla:cons_of_(*)}, we have $N(v)\setminus X = N_1(v)\setminus X$ and for every vertex $u\in P$ we have $Spe(u)\subseteq\{1,2,3\}$, thus $P \subseteq X$.

We next partition the set $N_1(v)\setminus X$ according to the spectrum of its vertices. Note that for every vertex $w\in N_1(v) \setminus X$ there exists at most one index $i\ge4$ such that $w\in \overline{V}(G_i[N_1(v)])$. Indeed, if $w,x,x'\in N_1(v)$ with $wx\in G_i$, $wx'\in G_{i'}$, for some $i,i'\geq 4$, $i\neq i'$, then $x'wxv$ is a rainbow $W_3$ using colors $i',i,1$ which can be extended to a rainbow $W_4$ using a neighbor $u$ of $v$ with $vu\in E(G_2)$. 

Let $A_i:= \overline{V}(G_i[N_1(v)])$, $i=4,\dots,9$.
Then
 $$N_1(v)\setminus X=A_4\cup A_5\cup A_6\cup A_7\cup A_8\cup A_9\cup R,$$
where $R$ is the set of the remaining vertices of $N_1(v)\setminus X$, that is, vertices $w$ with $m(v,w)=1$ and $w\notin \bigcup_{i=4}^9\overline{V}(G_i[N_1(v)])$ (see Fig.~\ref{F4}). In particular, no graph $G_i$, $i=4,\dots,9$, has an edge inside $R$ or between $A_i$ and $R$. 
Let $a_i=|A_i|$, $i=4,\dots,9$, and $r=|R|$.

Similarly, one can partition
$$V\setminus (X\cup N_1(v))=B_4\cup B_5\cup B_6\cup B_7\cup B_8\cup B_9\cup Z,$$
where 
\[B_i=\{u\in V\setminus (X\cup N(v)): Spe(u)\setminus\{1,2,3\}=\{i\} \},\] $i \geq 4$, while
\[Z=\{u\in V\setminus (X\cup N(v)): |Spe(u)\setminus\{1,2,3\}|\ge2 \}.\] 
Let $b_i=|B_i|$, $i=4,\dots,9$, and $z=|Z|$. 

\begin{figure}[ht]
\begin{center}
\begin{tikzpicture}[scale=1]

\draw[fill=black] (2,4.5) circle [radius=2pt];
\node at (2,4.5) [above] {$v$};

\draw [ultra thick] (2,4.5) -- (1,2.7);
\draw [ultra thick] (2,4.5) -- (4,2.7);
\draw [ultra thick] (2,4.5) -- (6,2.7);
\draw [ultra thick] (2,4.5) .. controls (1.6,4.9) and (0.1,4.9) .. (-0.3,4.5);
\draw [ultra thick, color=applegreen] (2,4.5) .. controls (1.6,4.1) and (0.1,4.1) .. (-0.3,4.5);

\draw (0,4.5) ellipse (3cm and 1.2cm);
\draw[fill=white] (-1,4.5) circle [radius=0.7];
\draw[fill=white] (1,2) circle [radius=0.7];
\draw[fill=white] (4,2) circle [radius=0.7];
\draw[fill=white] (6,2) circle [radius=0.7];

\node at (-2.7,3.3) [above] {$X$};
\node at (-1.6,3.5) [above] {$P$};
\node at (0.4,1) [above] {$A_4$};
\node at (2.5,2) {$\dots$};
\node at (3.4,1) [above] {$A_9$};
\node at (6.6,1) [above] {$R$};

\draw[fill=white] (1,-1) circle [radius=0.7];
\draw[fill=white] (4,-1) circle [radius=0.7];
\draw[fill=white] (6,-1) circle [radius=0.7];

\node at (0.4,-2) [above] {$B_4$};
\node at (2.5,-1) {$\dots$};
\node at (3.4,-2) [above] {$B_9$};
\node at (6.6,-2) [above] {$Z$};

\node at (5.4,3.1) [right] {1};
\node at (3.6,3.1) [right] {1};
\node at (1.2,3.1) [right] {1};

\draw[ultra thick, color=gptblue] (1,1.3) -- (1,-0.3);
\draw[ultra thick, color=gptblue] (6,1.3) -- (1,-0.3);
\draw[ultra thick, color=gptblue] (6,-1.7) .. controls (5.5,-2.5) and (1.5,-2.5) .. (1,-1.7);
\draw[ultra thick, color=gptblue] (0.6,2) -- (1.4,2);
\draw[ultra thick, color=gptblue] (0.6,-1) -- (1.4,-1);
\draw[ultra thick, color=gptblue] (5.6,-0.8) -- (6.4,-0.8);
\node at (1,2) [above] {\textcolor{gptblue}{4}};
\node at (1,-1) [above] {\textcolor{gptblue}{4}};
\node at (1,0.5) [left] {\textcolor{gptblue}{4}};
\node at (3,0.5) [left] {\textcolor{gptblue}{4}};
\node at (3.5,-2.3) [below] {\textcolor{gptblue}{4}};
\node at (6,-0.8) [above] {\textcolor{gptblue}{4}};

\draw[ultra thick, color=brilliantrose] (3.6,2) -- (4.4,2);
\draw[ultra thick, color=brilliantrose] (3.6,-1) -- (4.4,-1);
\draw[ultra thick, color=brilliantrose] (4,1.3) -- (4,-0.3);
\draw[ultra thick, color=brilliantrose] (6,1.3) -- (4,-0.3);
\draw[ultra thick, color=brilliantrose] (4.7,-1) -- (5.3,-1);
\draw[ultra thick, color=brilliantrose] (5.6,-1.2) -- (6.4,-1.2);

\node at (4,2) [above] {\textcolor{brilliantrose}{9}};
\node at (4,-1) [above] {\textcolor{brilliantrose}{9}};
\node at (4,0.5) [right] {\textcolor{brilliantrose}{9}};
\node at (5.1,0.5) [right] {\textcolor{brilliantrose}{9}};
\node at (5,-1) [below] {\textcolor{brilliantrose}{9}};
\node at (6,-1) {$\dots$};
\node at (6,-1.2) [below] {\textcolor{brilliantrose}{9}};

\node at (-1,4.5) [above] {$u$};
\node at (0.85,4.8) [above] {$1$};
\node at (0.85,4.2) [below,  color=applegreen] {$2$};

\draw[fill=black] (2,4.5) circle [radius=2pt];
\draw[fill=black] (-1,4.5) circle [radius=2pt];
    
\end{tikzpicture}
\caption{The partitions of $N_1(v)\setminus X$ and $V\setminus(N_1(v)\cup X)$ according to the type of spectrum. The thick, labeled edges correspond to possible edges of $M$. Each vertex in $P$ is adjacent to $v$ in at least two graphs from $G_1$, $G_2$, $G_3$.}\label{F4}
\end{center}
\end{figure}
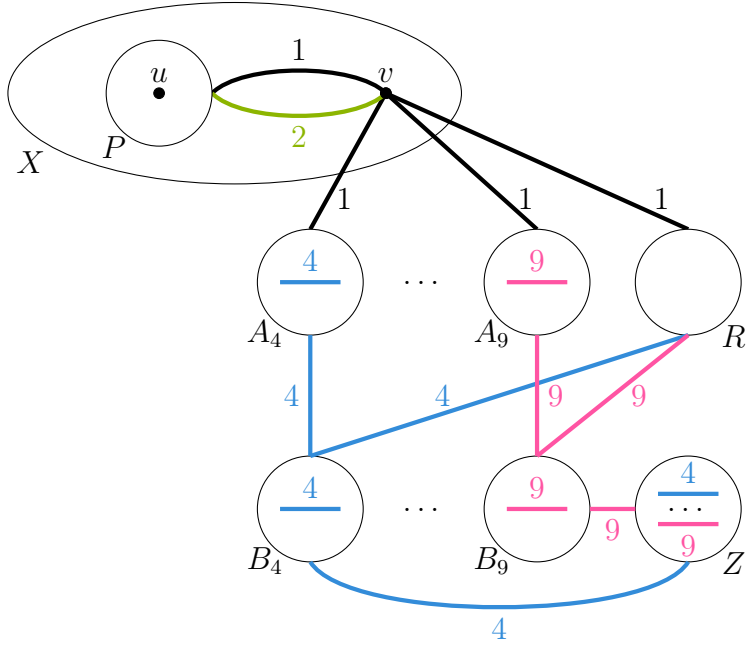

Note that if $wx \in E(G_i)$ for $w \in A_j$, $x \in B_h$ with $i,j,h \in \{4,\ldots,9\}$, then $i=j=h$ as otherwise it would create a rainbow $W_4$. Similarly, any edge $wx \in E(G_i)$ for $w \in N_1 (v) \setminus X$, $x \in Z$ with $i \in \{4,\ldots,9\}$ creates a rainbow $W_4$.  
Therefore, for every $i \geq 4$, the edges of $G_i$ can  only be contained in $A_i \cup B_i$, or in $Z$, or between $B_i$ and $R \cup Z$. This implies that, very crudely, for every $i\ge4$,
\begin{equation}\label{eGi}
e(G_i) \leq \frac{1}{2}(a_i+b_i+z)^2  + b_i r.
\end{equation}

Set $p:=p(v)$ and note  that $r +p+z\le n$, since the respective sets are disjoint. Also, by the definition of $b_i$'s and $z$, and by~Observation~\ref{3consec}\,(iii),
$$\sum_{i=4}^9 b_i +z \leq n-|N(v)|\le p.$$

W.l.o.g.~assume that $b_4=\min_{i \geq 4}b_i$. This and the above estimates imply that
\[b_4 \leq \frac{1}{6}\sum_{i\geq4} b_i \leq \frac{1}{6}(n-|N(v)|-z) \leq \frac{1}{6}(p-z).\]
It follows that $z \leq p$.
Moreover, for any $i\geq 4$, Claim~\ref{cla:9col_n/3} applied to $S=A_i\cup B_i\cup P$ yields that $a_i+b_i+p \leq \frac{n}{3}$. In particular, $p \leq \frac{n}{3}$ and $a_4+b_4+z \leq \frac{n}{3}-p+z$.

Finally, by~\eqref{eGi} with $i=4$,
\begin{align*}
e(G_4) &\leq \frac{1}{2}(a_4+b_4+z)^2 + b_4 r\\
&\leq \frac{1}{18}(n-3p+3z)^2 + \frac{1}{6}(p-z)(n-p-z)\\
&\leq \max\left\{\frac{1}{18}(n-3p)^2 + \frac{1}{6}p(n-p),\ \frac{1}{18}n^2\right\}\\
&\leq \max\left\{\frac{1}{18}n^2,\ \frac{1}{27}n^2,\ \frac{1}{18}n^2\right\} = \frac{1}{18}n^2,
\end{align*}
because the quadratic function of $z$ in the second line (treating $n$ and $p$ as parameters) is convex and attains its maximum for $z=0$ or $z=p$, while the quadratic function of $p$ in the third line, also convex, attains its maximum for $p=0$ or $p=n/3$. This contradicts the assumption on $e(G_4)$ and thus completes the proof of Theorem~\ref{thm:p5min_upper9}.
\end{proof}

\section{Rainbow $P_5$ in complete colored multi-graphs}\label{complete}

In this section, we prove Theorem~\ref{thm:P5_complete}. In view of the lower bound construction given in Section~\ref{res_min}, it suffices to prove the following  statement. 

\begin{theorem}\label{thm:P5_complete_upper}
For all $n\ge 21$, every complete colored multi-graph $M(G_1, G_2, G_3, G_4)$ on $n$ vertices, with 
\begin{equation}\label{ass}e(G_i)\geq
    \begin{cases}
        \binom{\lfloor n/3 \rfloor}2+1 & \mathrm{if}\quad n \not\equiv 2 \pmod 3\,,\\
        \binom{\lfloor n/3 \rfloor}2 + 2 & \mathrm{if}\quad n \equiv 2 \pmod 3\,.\\
    \end{cases}
    \end{equation}
for each $i \in\{1,2,3,4\}$, contains a rainbow $P_5$.
\end{theorem}

\begin{proof}
Suppose, to the contrary, that $M$ is a complete colored multi-graph on $V$ satisfying assumption~\eqref{ass} and containing no rainbow $P_5$. 
We will first eliminate certain rainbow subgraphs from $M$.

\begin{claim}\label{cla:complete-claw}
    $M$ contains no rainbow $K_{1,3}$.   
\end{claim}

\begin{proof}
To the contrary, w.l.o.g.~suppose that the colored multi-graph $M$ contains a subset of four vertices $U=\{u, v_1, v_2, v_3\}$ with $uv_i\in E(G_i)$, for $i \in \{1,2,3\}$. We consider two main cases depending on the location of the edges of $G_4$ with respect to $U$.

\smallskip
{\bf Case 1:} There is an edge $yz\in E(G_4)$ for some $y,z \in V\setminus U$.

\smallskip
{\bf Case 1a:} There exists a vertex $x \in V\setminus \{u,z\}$ such that $xy\in E(G_1)\cup E(G_2)\cup E(G_3)$, say $xy\in E(G_1)$. 
The absence of a rainbow $P_5$ in $M$ yields $x \not\in\{v_2,v_3\}$, $v_2y\in E(G_2)$, $v_3y\in E(G_3)$, and, in turn,  $v_1y\in E(G_1)$. Now consider the edge $v_1v_2$. In order not to complete a rainbow $P_5$, one must have $v_1v_2 \in E(G_1)\cup E(G_2)$. However, in either case, the edge $v_3z$ would complete a rainbow $P_5$ (see Fig.~\ref{Fig:claw}), a contradiction.  

\smallskip
{\bf Case 1b:} For all $x \in V\setminus \{u\}$ we have $xy\in E(G_4)$, and, by symmetry, for all $x \in V\setminus \{u\}$ we have $xz\in E(G_4)$. In particular, to avoid a rainbow $P_5$, necessarily $v_1v_2\not\in E(G_1) \cup E(G_2)$, $v_1v_3\not\in E(G_1) \cup E(G_3)$ and $v_2v_3\not\in E(G_2) \cup E(G_3)$. Suppose now that there exists an edge $e\in E(G_1)\cup E(G_2) \cup E(G_3)$, other than $v_1v_2, v_1v_3, v_2v_3$ and such that $e\cap\{u,y,z\}=\emptyset$. Let $v\in e\setminus \{v_1,v_2,v_3\}$. Then $vy\in E(G_4)$ and we are back in Case 1a (see Fig.~\ref{Fig:case1b} (left)). Thus, the only edges in $G_1, G_2, G_3$, other than those spanned by vertices in $U$, must be incident to $u$. Since by~\eqref{ass} the number of edges in each color is at least $n$, for every $v \in V \setminus\{u\}$ the edge $uv \in E(G_1)\cap E(G_2)\cap E(G_3)$ and $v_1v_2\in E(G_3)$, $v_1v_3\in E(G_2)$ and $v_2v_3\in E(G_1)$ (see Fig.~\ref{Fig:case1b} (right)). This leads to a rainbow $P_5$ with the vertex sequence $yv_1v_2uv_3$, a contradiction.  

\smallskip
{\bf Case 2:} All edges in $G_4$ are incident to $U$. 

\smallskip
{\bf Case 2a:} There exists $i\in\{1,2,3\}$ such that $|N_4(v_i)\setminus U| \geq 2$. Let $x,y\in N_4(v_i)\setminus U$, $x\neq y$. Then to avoid a rainbow $P_5$ we need $xy\in E(G_i)$. But then for $j \neq i$ any edge $v_jx$ leads to a rainbow $P_5$ (see Fig.~\ref{Fig:case2} (left)). 

\smallskip
{\bf Case 2b:} For each $i \in \{1,2,3\}$, $|N_4(v_i)\setminus U|\le 1$. First, suppose that there is a vertex $v_4\in N_4(u)\setminus U$. Take any pair of vertices $x,y$ in $V \setminus \{u,v_1,v_2,v_3,v_4\}$. Since $M$ is complete, $xy\in E(G_i)$ for some $i \in \{1,2,3,4\}$. But then $xy$ together with  $\{u,v_1,v_2,v_3,v_4\} \setminus \{v_i\}$ creates a structure already considered in Case 1 (possibly after reordering the colors; see Fig.~\ref{Fig:case2} (right)). 
Therefore, $N_4(u)\setminus U=\emptyset$, and thus $e(G_4)\leq 6 + 3$ which grossly contradicts assumption~\eqref{ass} for $i=4$.
\end{proof}

\begin{figure}[ht]
\begin{center}
\begin{tikzpicture}[scale=1]
    \node at (0,0) [left] {$u$};
    \node at (1.6,1.6) [above] {$v_1$};
    \node at (2,0) [below] {$v_2$};
    \node at (1.6,-1.6) [below] {$v_3$};
    \node at (5,0) [right] {$y$};
    \node at (3.4,1.6) [above] {$x$};
    \node at (3.4,-1.6) [below] {$z$};

    \node at (8,0) [left] {$u$};
    \node at (9.6,1.6) [above] {$v_1$};
    \node at (10,0) [below] {$v_2$};
    \node at (9.6,-1.6) [below] {$v_3$};
    \node at (13,0) [right] {$y$};
    \node at (11.4,1.6) [above] {$x$};
    \node at (11.4,-1.6) [below] {$z$};

    \node at (0.8,0.8) [above, color=applegreen] {1};
    \node at (4.2,0.8) [above, color=applegreen] {1};
    \node at (0.8,-0.8) [below, color=brilliantrose] {3};
    \node at (4.2,-0.8) [below, color=black] {4};
    \node at (1.2,0) [below, color=gptblue] {2};

    \node at (8.8,0.8) [above, color=applegreen] {1};
    \node at (11.2,0.8) [above, color=applegreen] {1};
    \node at (12.2,0.8) [above, color=applegreen] {1};
    \node at (8.8,-0.8) [below, color=brilliantrose] {3};
    \node at (11.2,-0.8) [below, color=brilliantrose] {3};
    \node at (12.2,-0.8) [below, color=black] {4};
    \node at (9.2,0) [below, color=gptblue] {2};
    \node at (11,0) [below, color=gptblue] {2};
    \node at (10.1,0.2) [above, color=applegreen] {1};
    \node at (10.5,0.2) [above] {or};
    \node at (10.9,0.2) [above, color=gptblue] {2};

    \draw[thick, color=applegreen] (0,0) -- (1.6,1.6) (5,0) -- (3.4,1.6);
    \draw[thick, color=gptblue] (0,0) -- (2,0);
    \draw[thick, color=brilliantrose] (0,0) -- (1.6,-1.6);
    \draw[thick] (5,0) -- (3.4,-1.6);

    \draw[thick, color=applegreen] (8,0) -- (9.6,1.6) (13,0) -- (11.4,1.6) (13,0) -- (9.6,1.6) (9.6,1.6) -- (10,0);
    \draw[thick, dashed, color=gptblue] (9.6,1.6) -- (10,0);
    \draw[thick, color=gptblue] (8,0) -- (10,0) (13,0) -- (10,0);
    \draw[thick, color=brilliantrose] (8,0) -- (9.6,-1.6) (13,0) -- (9.6,-1.6);
    \draw[thick] (13,0) -- (11.4,-1.6);
    \draw[thick, dashed] (9.6,-1.6) -- (11.4,-1.6);

    \draw[fill=black] (0,0) circle [radius=2pt];
    \draw[fill=black] (2,0) circle [radius=2pt];
    \draw[fill=black] (1.6,1.6) circle [radius=2pt];
    \draw[fill=black] (1.6,-1.6) circle [radius=2pt];
    \draw[fill=black] (5,0) circle [radius=2pt];
    \draw[fill=black] (3.4,1.6) circle [radius=2pt];
    \draw[fill=black] (3.4,-1.6) circle [radius=2pt];

    \draw[fill=black] (8,0) circle [radius=2pt];
    \draw[fill=black] (10,0) circle [radius=2pt];
    \draw[fill=black] (9.6,1.6) circle [radius=2pt];
    \draw[fill=black] (9.6,-1.6) circle [radius=2pt];
    \draw[fill=black] (13,0) circle [radius=2pt];
    \draw[fill=black] (11.4,1.6) circle [radius=2pt];
    \draw[fill=black] (11.4,-1.6) circle [radius=2pt];

\end{tikzpicture}
\caption{In Case 1a, a rainbow $K_{1,3}$ forces a rainbow $P_5$; note that possibly $x=v_1$.}\label{Fig:claw}
\end{center}
\end{figure}
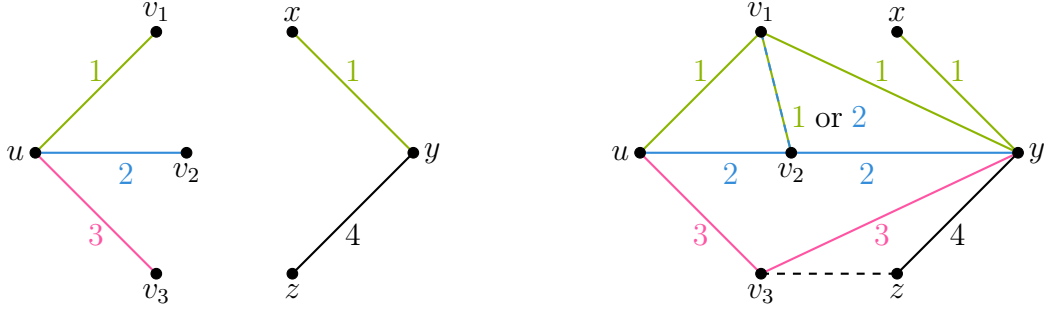

\begin{figure}[ht]
\centering
\begin{tikzpicture}[
    vertex/.style={circle, fill=black, inner sep=2pt},
    edge/.style={thick},
    c1/.style={thick, applegreen},
    c2/.style={thick, gptblue},
    c3/.style={thick, brilliantrose},
    c4/.style={thick, black},
]


\coordinate (Lu)  at (0,0);
\coordinate (Lv1) at (2.4,1.6);
\coordinate (Lv2) at (1.8,0);
\coordinate (Lv3) at (2.4,-1.6);
\coordinate (Lv)  at (4,1.6);
\coordinate (La)  at (5.8,1.6);
\coordinate (Ly)  at (5.6,0);
\coordinate (Lz)  at (4,-1.6);

\node[left]  at (Lu)  {$u$};
\node[above] at (Lv1) {$v_1$};
\node[below, xshift=-1.5mm] at (Lv2) {$v_2$};
\node[below] at (Lv3) {$v_3$};
\node[above] at (Lv)  {$v$};
\node[right] at (Ly)  {$y$};
\node[below] at (Lz)  {$z$};

\draw[c1] (Lu)  -- node[midway, above, applegreen] {1} (Lv1);
\draw[c2] (Lu)  -- node[midway, above, gptblue] {\;2} (Lv2);
\draw[c3] (Lu)  -- node[midway, below, brilliantrose] {3} (Lv3);

\draw[c4] (Ly) -- node[midway, below] {4} (Lv);
\draw[c4] (Ly) -- node[midway, below] {4} (Lz);
\draw[c1] (Lv) -- node[midway, above] {1} (La);

\foreach \p in {Lu,Lv1,Lv2,Lv3,Lv,Ly,Lz,La}
    \node[vertex] at (\p) {};


\coordinate (Ru)  at (8,0);
\coordinate (Rv1) at (10.4,1.6);
\coordinate (Rv2) at (9.8,0);
\coordinate (Rv3) at (10.4,-1.6);
\coordinate (Ry)  at (13.6,0);
\coordinate (Rz)  at (12,-1.6);

\node[left]  at (Ru)  {$u$};
\node[above] at (Rv1) {$v_1$};
\node[below, xshift=-1.5mm] at (Rv2) {$v_2$};
\node[below] at (Rv3) {$v_3$};
\node[above] at (Ry)  {$y$};
\node[below] at (Rz)  {$z$};

\newcommand{\triedge}[3]{
  \draw[applegreen,    line width=3.0pt, line cap=round] (#1) -- (#2);
  \draw[gptblue,       line width=2.0pt, line cap=round] (#1) -- (#2);
  \draw[brilliantrose, line width=1.0pt, line cap=round]
    (#1) -- (#2)
    node[midway, sloped, #3] {
      \textcolor{applegreen}{1}\,\!
      \textcolor{gptblue}{2}\,\!
      \textcolor{brilliantrose}{3}
    };
}

\triedge{Ru}{Rv1}{above=2pt}
\triedge{Ru}{Rv2}{above=0pt, xshift=10pt}
\triedge{Ru}{Rv3}{below=2pt}

\draw[c1] (Rv2)  -- node[midway, left, applegreen] {1} (Rv3);
\draw[c2] (Rv1)  -- node[midway, right, gptblue] {2} (Rv3);
\draw[c3] (Rv1)  -- node[midway, left, brilliantrose] {3} (Rv2);

\draw[c3] (Rv1) -- (Rv2);
\draw[c1] (Rv2) -- (Rv3);
\draw[c2] (Rv1) -- (Rv3);

\draw[c4] (Ry) -- node[midway, below] {4} (Rv1);
\draw[c4] (Ry) -- node[midway, below] {4} (Rz);

\foreach \p in {Ru,Rv1,Rv2,Rv3,Ry,Rz}
    \node[vertex] at (\p) {};

\end{tikzpicture}

\caption{In Case 1b, a rainbow $K_{1,3}$ forces a rainbow $P_5$.}
\label{Fig:case1b}
\end{figure}
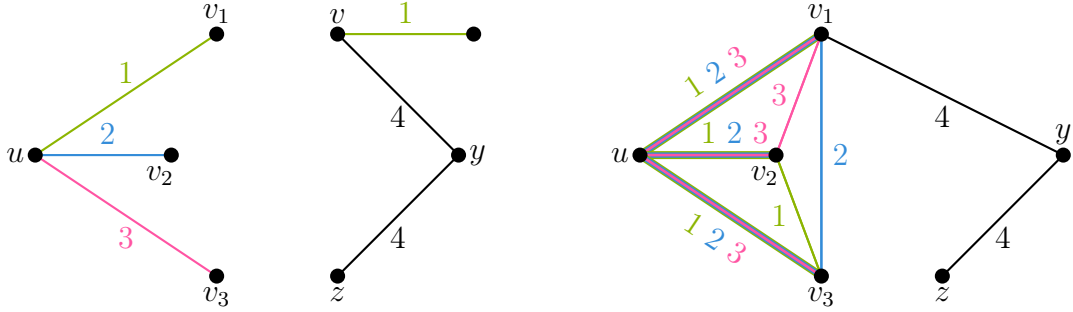

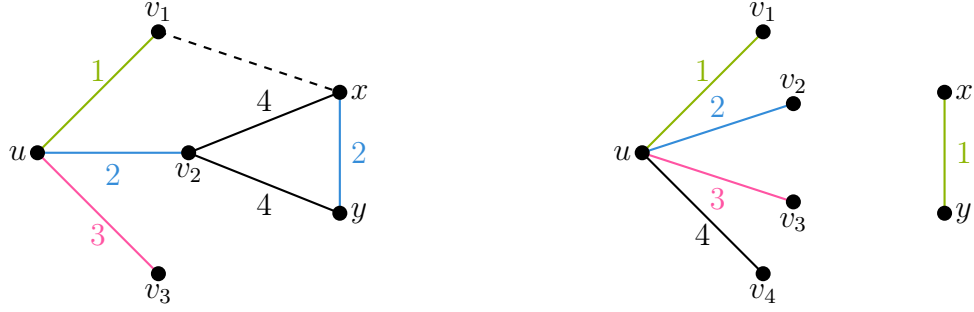
\begin{figure}[ht]
\centering
\begin{tikzpicture}[
    vertex/.style={circle, fill=black, inner sep=2pt},
    edge/.style={thick},
    c1/.style={thick, applegreen},
    c2/.style={thick, gptblue},
    c3/.style={thick, brilliantrose},
    c4/.style={thick, black},
    c5/.style={thick, gptorange},
]


\coordinate (Lu)  at (0,0);
\coordinate (Lv1) at (1.6,1.6);
\coordinate (Lv2) at (2,0);
\coordinate (Lv3) at (1.6,-1.6);
\coordinate (Lx)  at (4,0.8);
\coordinate (Ly)  at (4,-0.8);

\node[left]  at (Lu)  {$u$};
\node[above] at (Lv1) {$v_1$};
\node[below] at (Lv2) {$v_2$};
\node[below] at (Lv3) {$v_3$};
\node[right] at (Lx)  {$x$};
\node[right] at (Ly)  {$y$};

\draw[c1] (Lu)  -- node[midway, above, applegreen] {1} (Lv1);
\draw[c2] (Lu)  -- node[midway, below, gptblue] {2} (Lv2);
\draw[c3] (Lu)  -- node[midway, below, brilliantrose] {3} (Lv3);
\draw[c2] (Ly) -- node[midway, right, gptblue] {2} (Lx);
\draw[c4] (Lv2) -- node[midway, above] {4} (Lx);
\draw[c4] (Lv2) -- node[midway, below] {4} (Ly);
\draw[c4, dashed] (Lv1) -- (Lx);

\foreach \p in {Lu,Lv1,Lv2,Lv3,Lx,Ly}
    \node[vertex] at (\p) {};


\coordinate (Ru)  at (8,0);
\coordinate (Rv1) at (9.6,1.6);
\coordinate (Rv2) at (10,0.65);
\coordinate (Rv3) at (10,-0.65);
\coordinate (Rv4) at (9.6,-1.6);
\coordinate (Rx)  at (12,0.8);
\coordinate (Ry)  at (12,-0.8);

\node[left]  at (Ru)  {$u$};
\node[above] at (Rv1) {$v_1$};
\node[above] at (Rv2) {$v_2$};
\node[below] at (Rv3) {$v_3$};
\node[below] at (Rv4) {$v_4$};
\node[right] at (Rx)  {$x$};
\node[right] at (Ry)  {$y$};

\draw[c1] (Ru)  -- node[midway, above, applegreen] {1} (Rv1);
\draw[c2] (Ru)  -- node[midway, above, gptblue] {2} (Rv2);
\draw[c3] (Ru)  -- node[midway, below, brilliantrose] {3} (Rv3);
\draw[c4] (Ru)  -- node[midway, below, black] {4} (Rv4);

\draw[c1] (Ry)  -- node[midway, right, applegreen] {1} (Rx);

\foreach \p in {Ru,Rv1,Rv2,Rv3,Rv4,Rx,Ry}
    \node[vertex] at (\p) {};

\end{tikzpicture}

\caption{Illustration to  Cases 2a and 2b in the proof of Claim~\ref{cla:complete-claw}. }
\label{Fig:case2}
\end{figure}

In the forthcoming argument it will come in handy to call a rainbow path $P_3$ in colors $i$ and $j$ an \emph{$ij$-path}. 

\begin{claim}\label{cherries}
For every $i\in\{1,2,3,4\}$ there is in $M$ a rainbow path $P_3$ using color $i$. Moreover, if for $i,j,h,\ell\in\{1,2,3,4\}$, $M$ contains an $ij$-path and an $h\ell$-path, then  $\{i,j\}\cap\{h,\ell\}\neq\emptyset$.
\end{claim}

\begin{proof}
Fix $i\in\{1,2,3,4\}$. If there were no rainbow $P_3$'s using color $i$, then, by completeness of $M$, for every $v\in V$ one would have $d_i(v)=n-1$ or $d_i(v)=0$. Since by~\eqref{ass} $e(G_i)>0$, it follows that $G_i=K_n$ and then any edge in another color forms a rainbow $P_3$ using color $i$ (in fact, many). 

For the second part, w.l.o.g.~suppose that in $M$ there are a 12-path $abc$ and a $34$-path $xyz$. 
We consider cases with respect to how much the two vertex sets overlap.

If $y=b$, then we have a rainbow $K_{1,3}$ contradicting Claim~\ref{cla:complete-claw}, unless $x=a$ and $z=c$. But then, for any vertex $w$, any color of the edge $yw$ again gives a rainbow $K_{1,3}$. 

If $y=a$, then we have a rainbow $K_{1,3}$ contradicting Claim~\ref{cla:complete-claw}, unless $x=b$. Now, to avoid a rainbow $K_{1,3}$, all edges incident to $b$, except $ab$ should belong to $G_2$, while all edges incident to $a$, except $ab$ should belong to $G_4$. But then, there is an edge $uv \in E(G_1)$ different from $ab$ and $cz$, creating a rainbow $P_5$, namely $uvbaz$ (if $z \not\in\{u,v\}$) or $uvabc$ (if $c \not\in\{u,v\}$). 

Therefore, by symmetry, we may assume that vertices $b$ and $y$ are different from $a,c,x,z$. Now, any color of the edge $by$ creates a rainbow $K_{1,3}$ contradicting Claim~\ref{cla:complete-claw}.
\end{proof}

Let us analyze the consequences of Claim~\ref{cherries}. W.l.o.g.~assume there is a $12$-path and a $13$-path in $M$. This implies that there is no $24$-path and no $34$-path, so there must be a $14$-path as well. The presence of these three rainbow paths excludes the presence of any other rainbow paths of length two.

\begin{corollary}\label{no}  There are no $23$-paths, $24$-paths, and $34$-paths in $M$.
\end{corollary}

All of this means that the structure of $M$ is somewhat special. 

\begin{observation}
There exists a partition
$V = V_1 \cup V_{1,2} \cup V_{1,3} \cup V_{1,4} \cup S$, where
\begin{itemize}
    \item $V_{1} = \{v\in V : Spe(v) =\{1\}\}$, 
    \item $V_{1,j} = \{v\in V\setminus V_1: Spe(v) = \{1,j\}\}$, for $j\in\{2,3,4\}$, 
    \item $M[S]\setminus G_1$ is a matching of multiple edges (with multiplicity $2$ or $3$) in colors $2,3$, and~$4$,
\end{itemize} 
and all edges between the above five sets have only color 1.
\end{observation}

\begin{proof}
First observe that there are no vertices $v$ with $Spe(v)=\{j\}$ for $j\in\{2,3,4\}$. Indeed, suppose that, say, $Spe(v)=\{2\}$. Then, by Corollary~\ref{no}, for all $v'\neq v$ one has $vv'\in E(G_2)$. Since $G_3\neq\emptyset$, there is an edge $uw\in E(G_3)$ with $v\notin\{u,w\}$. But then $vuw$ is a $23$-path, which contradicts Corollary~\ref{no}. 

Therefore, the set $S=V\setminus ( V_1 \cup V_{1,2} \cup V_{1,3} \cup V_{1,4})$ consists of all vertices $v$ such that $|Spe(v)\cap\{2,3,4\}|\geq 2$. If $v \in S$, say $\{2,3\} \subseteq Spe(v)$, then, by Corollary~\ref{no}, there exists a vertex $w \in S$ such that $vw \in E(G_2) \cap E(G_3)$, and for any $u\in V  \setminus\{v,w\}$ one has $uv, uw \not\in E(G_2) \cup E(G_3) \cup E(G_4)$.
Hence, $M[S]\setminus G_1$ is a matching consisting of multiple edges in colors $2,3$, and~$4$. Moreover, for each $j\in\{2,3,4\}$, every edge of $G_j$ is either an edge of a matching inside $S$ or has both endpoints in $V_{1,j}$.
\end{proof}

Now, we can quickly finish the proof of Theorem~\ref{thm:P5_complete_upper}. 
W.l.o.g.~we may assume that  $|V_{1,2}|=\min\{|V_{1,2}|,|V_{1,3}|,|V_{1,4}|\}$. Since the set $S$ induces a matching, $s = |S|$ is even. We thus have
    \begin{equation}\label{eq:complete-Sbound}
        e(G_2) \leq \frac s2 + \binom{|V_{1,2}|}{2} \leq \frac s2 + \binom{\lfloor\frac{n-s}3\rfloor}{2}.
    \end{equation}
    If $s=0$, then (\ref{eq:complete-Sbound}) reduces to $e(G_2) \leq \binom{\lfloor n/3 \rfloor}{2}$, contradicting~\eqref{ass}.
    If $s=2$, then we also obtain $e(G_2) \leq1+\binom{\lfloor\frac{n-2}3\rfloor}{2}\leq \binom{\lfloor n/3 \rfloor}{2}$, unless $\lfloor\frac{n-2}3\rfloor = \lfloor\frac n3\rfloor$, that is, unless $n \equiv 2 \pmod 3$. But then $e(G_2) \leq 1 + \binom{\lfloor n/3 \rfloor}{2}$, which contradicts~\eqref{ass} in this case.
    For $s \geq 4$, (\ref{eq:complete-Sbound}) gives 
    \[e(G_2) \leq \frac{s}{2} + \binom{\frac{n-s}{3}}{2} = \frac{n^2-2ns+s^2-3n+12s}{18} \leq \frac{\max\{n^2-11n+64, 9n\}}{18}\leq\binom{\frac{n-2}{3}}{2},  \]
    because the expression in the middle  is a convex function of $s$, $4 \le s \le n$, while the expression on the right can be rewritten as $(n^2-7n+10)/18$. 
    But
    $\binom{\frac{n-2}{3}}{2}\leq \binom{\lfloor n/3 \rfloor}{2},$
    and one gets a contradiction with~\eqref{ass}.
\end{proof}

\section{The Conjecture}\label{fin_rem}

\subsection{Four colors}\label{conj4}

At the moment, we are unable to pinpoint the threshold $\pi_4^{\min}(P_5)$. The lower bound~\eqref{kell2}, which has determined the values of $\pi_k^{\min}(P_5)$ for $k\in\{5,6,9\}$, is not optimal for $k=4$. Indeed, consider the following construction. Split the  vertex set into disjoint sets $A_1$, $A_2$, $A_3$, each of size $\lfloor (3-2\sqrt{2})n \rfloor$, and the remaining set $B$  (of size, approximately, $(6\sqrt2-8)n$). For $i \in \{1,2,3\}$, let $G_i$ be the union of two complete graphs, one induced on $A_i$, the other on $B$, and let $G_4$ be the complete graph on $A_1 \cup A_2 \cup A_3$ (see Figure~\ref{Fig:4col}). Then, the number of edges in each of these four graphs is approximately
\begin{align*}
\left(153-108 \sqrt{2}\right)\binom n2 + O(n) \approx 0.265\binom n2,
\end{align*}
which yields that $\pi_4^{\min}(P_5)>\frac14$. We speculate that this construction does determine the rainbow Tur\'an density for $P_5$.

\begin{conjecture}
We have $\pi_4^{\min}(P_5)=153-108 \sqrt{2}$.
\end{conjecture}

\begin{figure}[ht]
\begin{center}
\begin{tikzpicture}[scale=1]

\draw[fill=white] (0,-1) circle [radius=0.7];
\draw[fill=white] (3,-1) circle [radius=0.7];
\draw[fill=white] (1.5,1) circle [radius=0.7];
\draw[fill=white] (7,0) circle [radius=1.1];

\draw [ultra thick] (0.7,-1) -- (2.3,-1);
\draw [ultra thick] (0.45,-0.45) -- (1.05,0.45);
\draw [ultra thick] (2.55,-0.45) -- (1.95,0.45);

\draw[thick, color=brilliantrose] (0,-1.5) .. controls (-0.2,-1.3) and (-0.2,-0.7) .. (0,-0.5);
\draw[thick] (0,-1.5) .. controls (0.2,-1.3) and (0.2,-0.7) .. (0,-0.5);
\draw[thick, color=applegreen] (1.5,0.5) .. controls (1.3,0.7) and (1.3,1.3) .. (1.5,1.5);
\draw[thick] (1.5,0.5) .. controls (1.7,0.7) and (1.7,1.3) .. (1.5,1.5);
\draw[thick, color=gptblue] (3,-1.5) .. controls (2.8,-1.3) and (2.8,-0.7) .. (3,-0.5);
\draw[thick] (3,-1.5) .. controls (3.2,-1.3) and (3.2,-0.7) .. (3,-0.5);
\draw[thick, color=brilliantrose] (7,-0.8) .. controls (6.4,-0.4) and (6.4,0.4) .. (7,0.8);
\draw[thick, color=applegreen] (7,-0.8) -- (7,0.8);
\draw[thick, color=gptblue] (7,-0.8) .. controls (7.6,-0.4) and (7.6,0.4) .. (7,0.8);

\node at (-0.5,-1.5) [below, left] {$A_1$};
\node at (2,1.5) [above, right] {$A_2$};
\node at (3.5,-1.5) [below, right] {$A_3$};
\node at (7.8,-0.8) [below, right] {$B$};
\node at (1.5,-1) [below] {$4$};
\node at (0.8,0.1) [left] {$4$};
\node at (2.2,0.1) [right] {$4$};
\node at (-0.1,-1) [left, color=brilliantrose] {$1$};
\node at (0.1,-1) [right] {$4$};
\node at (1.4,1) [left, color=applegreen] {$2$};
\node at (1.6,1) [right] {$4$};
\node at (2.9,-1) [left, color=gptblue] {$3$};
\node at (3.1,-1) [right] {$4$};
\node at (6.65,0) [left, color=brilliantrose] {$1$};
\node at (6.95,0) [right, color=applegreen] {$2$};
\node at (7.35,0) [right, color=gptblue] {$3$};

\end{tikzpicture}
\caption{A construction showing $\pi_4^{\min}(P_5)\geq153-108 \sqrt{2}\geq0.264$.}\label{Fig:4col}
\end{center}
\end{figure}
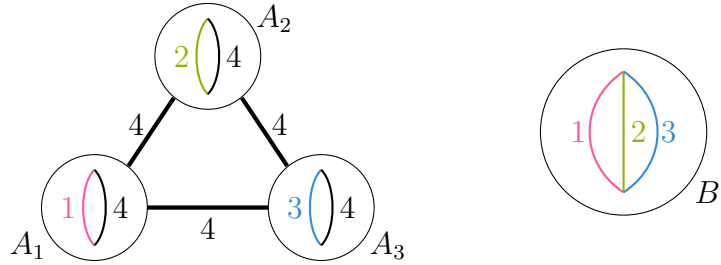

\subsection{General conjecture}\label{conjk}

The construction believed to be optimal for $k=4$ can be generalized to all numbers of colors $k$ such that $k\equiv 1 \pmod 3$. 

\smallskip
{\bf Construction I.} Let $k=3h+1$, $h\ge1$, and split $V$ into $h+3$ sets 
$$V=A_1\cup A_2\cup A_3\cup B\cup C_1\cup\cdots\cup C_{h-1} $$ of sizes 
$$|A_1|=|A_2|=|A_3|=\left\lfloor\frac{3h-2\sqrt{2}}{9h^2-8}n\right\rfloor,\quad |C_1|=\cdots=|C_{h-1}|=\left\lfloor\frac{9h-6\sqrt{2}}{9h^2-8}n\right\rfloor,$$
and $$|B|\approx\frac{6\sqrt{2}h-8}{9h^2-8}n.$$ 
(Note that for $k=4$ one has $h=1$ and this partition reduces to that in Subsection~\ref{conj4}.)

Then, for each $i \in [h-1]$, let $G_{i}, G_{h+i}, G_{2h+i}$ be the complete graphs on vertex set $C_i$, and let the remaining graphs $G_h, G_{2h}, G_{3h}$ and $G_{3h+1}$ form the construction for $4$ colors presented in Subsection~\ref{conj4}, that is, each of $G_h, G_{2h}, G_{3h}$ is the union of two complete graphs, one  on $B$ and the other one on, respectively,  $A_1, A_2, A_3$, while $G_{3h+1}$ is the complete graph on $A_1 \cup A_2 \cup A_3$ (see Figure~\ref{Fig:7col} for the case $k=7$).

\begin{figure}[ht]
\begin{center}
\begin{tikzpicture}[scale=1]

\draw[fill=white] (0,-1) circle [radius=0.7];
\draw[fill=white] (3,-1) circle [radius=0.7];
\draw[fill=white] (1.5,1) circle [radius=0.7];
\draw[fill=white] (6,0) circle [radius=1.1];
\draw[fill=white] (9.5,0) circle [radius=1.4];

\draw [ultra thick] (0.7,-1) -- (2.3,-1);
\draw [ultra thick] (0.45,-0.45) -- (1.05,0.45);
\draw [ultra thick] (2.55,-0.45) -- (1.95,0.45);

\draw[thick, color=brilliantrose] (0,-1.5) .. controls (-0.2,-1.3) and (-0.2,-0.7) .. (0,-0.5);
\draw[thick] (0,-1.5) .. controls (0.2,-1.3) and (0.2,-0.7) .. (0,-0.5);
\draw[thick, color=applegreen] (1.5,0.5) .. controls (1.3,0.7) and (1.3,1.3) .. (1.5,1.5);
\draw[thick] (1.5,0.5) .. controls (1.7,0.7) and (1.7,1.3) .. (1.5,1.5);
\draw[thick, color=gptblue] (3,-1.5) .. controls (2.8,-1.3) and (2.8,-0.7) .. (3,-0.5);
\draw[thick] (3,-1.5) .. controls (3.2,-1.3) and (3.2,-0.7) .. (3,-0.5);
\draw[thick, color=brilliantrose] (6,-0.8) .. controls (5.4,-0.4) and (5.4,0.4) .. (6,0.8);
\draw[thick, color=applegreen] (6,-0.8) -- (6,0.8);
\draw[thick, color=gptblue] (6,-0.8) .. controls (6.6,-0.4) and (6.6,0.4) .. (6,0.8);
\draw[thick, color=gptorange] (9.5,-0.9) .. controls (8.9,-0.5) and (8.9,0.5) .. (9.5,0.9);
\draw[thick, color=gptpurple] (9.5,-0.9) -- (9.5,0.9);
\draw[thick, color=gptgold] (9.5,-0.9) .. controls (10.1,-0.5) and (10.1,0.5) .. (9.5,0.9);

\node at (-0.5,-1.5) [below, left] {$A_1$};
\node at (2,1.5) [above, right] {$A_2$};
\node at (3.5,-1.5) [below, right] {$A_3$};
\node at (6.7,-0.9) [below, right] {$B$};
\node at (10.6,-0.9) [below, right] {$C_1$};

\node at (1.5,-1) [below] {$7$};
\node at (0.8,0.1) [left] {$7$};
\node at (2.2,0.1) [right] {$7$};
\node at (-0.1,-1) [left, color=brilliantrose] {$2$};
\node at (0.1,-1) [right] {$7$};
\node at (1.4,1) [left, color=applegreen] {$4$};
\node at (1.6,1) [right] {$7$};
\node at (2.9,-1) [left, color=gptblue] {$6$};
\node at (3.1,-1) [right] {$7$};
\node at (5.65,0) [left, color=brilliantrose] {$2$};
\node at (5.95,0) [right, color=applegreen] {$4$};
\node at (6.35,0) [right, color=gptblue] {$6$};
\node at (9.15,0) [left, color=gptorange] {$1$};
\node at (9.45,0) [right, color=gptpurple] {$3$};
\node at (9.85,0) [right, color=gptgold] {$5$};

\end{tikzpicture}
\caption{Construction I from Conjecture~\ref{conjecture} for $7$ colors. 
}\label{Fig:7col}
\end{center}
\end{figure}
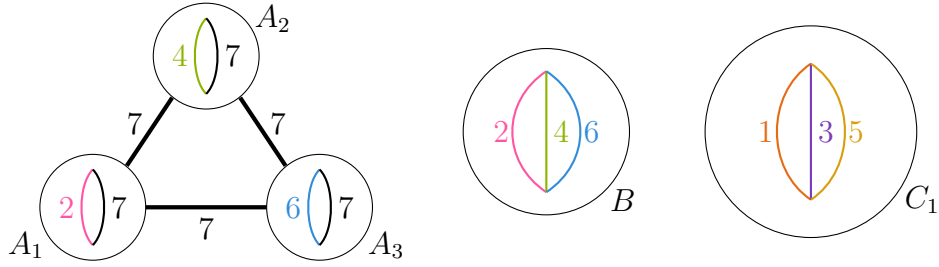

This construction yields the bound
$$\pi_{3h+1}^{\min}(P_5)\ge \left(\frac{9h-6\sqrt{2}}{9h^2-8}\right)^2,$$
which turns out to be better than bound~\eqref{kell2}  stemming from a standard construction (cf. Subsection~\ref{res_min}) which we recall next.

\smallskip
{\bf Construction II.} For every $k\ge 4$, split $V$ into $\lceil \frac{k}{3} \rceil$ sets $V_j$ of (almost) equal size and induce cliques in three colors on each of these sets. If $k=3h+r$, $1\le r\le 2$, then one set $V_j$ gets only $r$ cliques (see Figure~\ref{Fig:8col}). This construction yields the bound

$$\pi_k^{\min}(P_5)\ge\left\lceil \frac{k}{3} \right\rceil^{-2}.$$
It is easy to check that for all $h\ge1$,
\[\left(\frac{9h-6\sqrt{2}}{9h^2-8}\right)^2 > \frac{1}{(h+1)^2},\]
so Construction I is better than Construction II for every $k \equiv 1 \pmod 3$. 

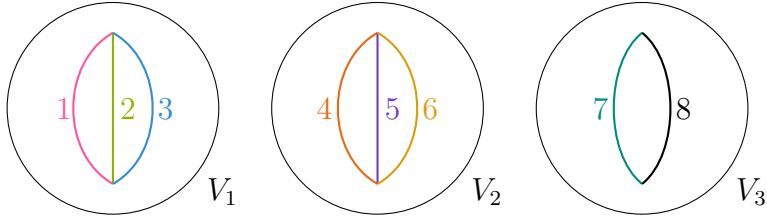
\begin{figure}[ht]
\begin{center}
\begin{tikzpicture}[scale=1]

\draw[fill=white] (0,0) circle [radius=1.4];
\draw[fill=white] (3.5,0) circle [radius=1.4];
\draw[fill=white] (7,0) circle [radius=1.4];

\draw[thick, color=brilliantrose] (0,-1) .. controls (-0.7,-0.6) and (-0.7,0.6) .. (0,1);
\draw[thick, color=applegreen] (0,-1) -- (0,1);
\draw[thick, color=gptblue] (0,-1) .. controls (0.7,-0.6) and (0.7,0.6) .. (0,1);
\draw[thick, color=gptorange] (3.5,-1) .. controls (2.8,-0.6) and (2.8,0.6) .. (3.5,1);
\draw[thick, color=gptpurple] (3.5,-1) -- (3.5,1);
\draw[thick, color=gptgold] (3.5,-1) .. controls (4.2,-0.6) and (4.2,0.6) .. (3.5,1);
\draw[thick, color=gptteal] (7,-1) .. controls (6.5,-0.6) and (6.5,0.6) .. (7,1);
\draw[thick, color=black] (7,-1) .. controls (7.5,-0.6) and (7.5,0.6) .. (7,1);

\node at (1.1,-1.1) [below, right] {$V_1$};
\node at (4.6,-1.1) [below, right] {$V_2$};
\node at (8.1,-1.1) [below, right] {$V_3$};

\node at (-0.4,0) [left, color=brilliantrose] {$1$};
\node at (-0.05,0) [right, color=applegreen] {$2$};
\node at (0.45,0) [right, color=gptblue] {$3$};
\node at (3.05,0) [left, color=gptorange] {$4$};
\node at (3.45,0) [right, color=gptpurple] {$5$};
\node at (3.95,0) [right, color=gptgold] {$6$};
\node at (6.7,0) [left, color=gptteal] {$7$};
\node at (7.3,0) [right, color=black] {$8$};
\end{tikzpicture}
\caption{Construction II from Conjecture~\ref{conjecture} for $8$ colors.}\label{Fig:8col}
\end{center}
\end{figure}

Still, in the long run, the real winner is, surprisingly, a different construction built upon an optimal construction for $P_4$ and $k\ge 7$ colors from \cite{BG23} (cf. Theorem~\ref{thm:p4min} in Section~\ref{res_min}).

\smallskip
{\bf Construction III.} 
Split $V$ into sets $A_0$ of size, approximately, $x_0n$, $A_1, A_2, \ldots, A_a$ of size, approximately, $xn$, $B_0$ of size, approximately, $y_0n$, and $B_1, B_2, \dots, B_{k-a-1}$ of size, approximately $yn$, where $ax+(k-a-1)y+x_0+y_0=1$. For $i\in [a]$, let $G_i$ be the complete graph on vertex set $A_i$, for $j \in [k-a-1]$, let $G_{a+j}$ be the union of the complete graph on $B_j$ and the complete bipartite graph between $B_j$ and $\bigcup_{i=0}^a A_i$ and let $G_k$ be the union of the complete graphs on $A_0$ and $B_0$ and the complete bipartite graph between $B_0$ and $\bigcup_{i=0}^a A_i$ (see Figure~\ref{Fig:constructionIII}).

The optimal bound for $\pi_k^{\min}(P_5)$ stemming from this construction can be computed by maximizing $x^2$ over integers $a \in [k-1]$ and non-negative reals $x, x_0, y, y_0$ under constraints $ax+(k-a-1)y+x_0+y_0=1$ and $x^2=y^2+2(ax+x_0)y=x_0^2+y_0^2+2(ax+x_0)y_0$. 
For example, for $k=10$, the best value is obtained for $a=2$, $x_0=0$, $x = \frac{7 + 4\sqrt{5}}{62}$ and $y_0=y = \frac{6-\sqrt{5}}{62}$ (or, equivalently, $a=1$, $x_0=x$, and $y_0=0$), and one gets $\pi_{10}^{\min}(P_5) \geq 0.0661$, while Constructions I and II yield only, respectively, $\pi_{10}^{\min}(P_5) \geq 0.0643$ and $\pi_{10}^{\min}(P_5) \geq 0.0625$.

In order to obtain the best bound in Construction III, the special graph $G_k$ may indeed need $x_0>0$ and $y_0>0$. For instance, for $k=13$, this construction yields $\pi_{13}^{\min}(P_5) \geq 0.0479$, while assuming $x_0=0$ we can obtain only $\pi_{13}^{\min}(P_5) \geq 0.0476$.

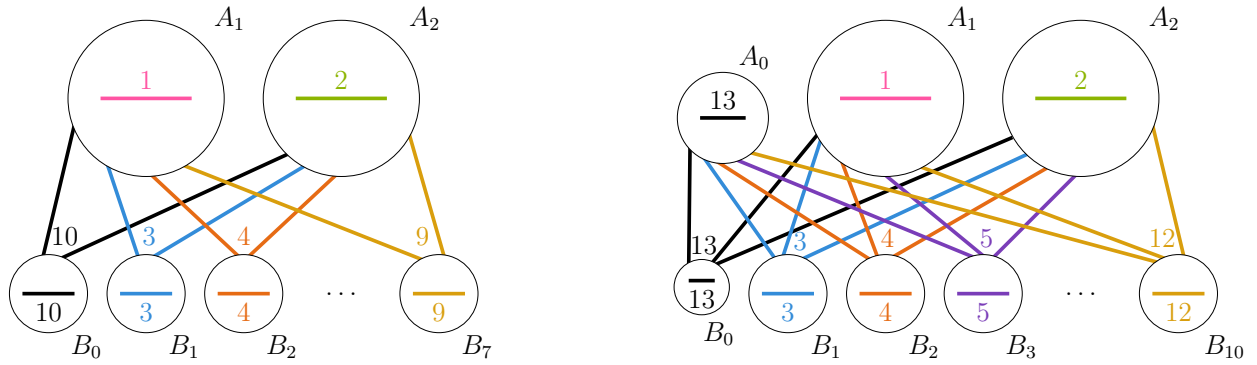
\begin{figure}[ht]
\centering

\begin{minipage}[c]{0.4\textwidth}
\centering
\resizebox{\linewidth}{!}{%
\begin{tikzpicture}[
    edge/.style  = {ultra thick}
]

\coordinate (A1) at (1.5,3);
\coordinate (A2) at (4.5,3);

\coordinate (B0) at (0,0);
\coordinate (B1) at (1.5,0);
\coordinate (B2) at (3,0);
\coordinate (B7) at (6,0);

\draw[edge, black] (-0.1,0.5) -- (0.4,2.6);
\draw[edge, black] ( 0.1,0.5) -- (3.8,2.2);

\draw[edge, gptblue] (1.4,0.5) -- (0.9,2.0);
\draw[edge, gptblue] (1.5,0.5) -- (4,2);

\draw[edge, gptorange] (3.0,0.5) -- (1.5,1.9);
\draw[edge, gptorange] (3.0,0.5) -- (4.5,1.9);

\draw[edge, gptgold] (5.8,0.5) -- (2.0,2.0);
\draw[edge, gptgold] (6.1,0.5)  -- (5.5,2.6);

\draw[fill=white] (A1) circle[radius=1.2];
\draw[fill=white] (A2) circle[radius=1.2];

\foreach \B in {B0,B1,B2,B7}
    \draw[fill=white] (\B) circle[radius=0.6];

\node at (4.5,0) {$\ldots$};

\draw[edge, brilliantrose] (0.8,3) -- (2.2,3);
\draw[edge, applegreen]    (3.8,3) -- (5.2,3);

\draw[edge, black]   (-0.4,0) -- (0.4,0);
\draw[edge, gptblue] (1.1,0)  -- (1.9,0);
\draw[edge, gptorange] (2.6,0)  -- (3.4,0);
\draw[edge, gptgold]   (5.6,0)  -- (6.4,0);

\node[above right]  at (2.4,3.9)  {$A_1$};
\node[above right] at (5.4,3.9)  {$A_2$};

\node[below right] at (0.2,-0.5) {$B_0$};
\node[below right] at (1.7,-0.5) {$B_1$};
\node[below right] at (3.2,-0.5) {$B_2$};
\node[below right] at (6.2,-0.5) {$B_7$};

\node[above, brilliantrose] at (A1) {$1$};
\node[above, applegreen]    at (A2) {$2$};

\node[below, black]   at (B0) {$10$};
\node[above, black]   at (0.25,0.6) {$10$};

\node[below, gptblue] at (B1) {$3$};
\node[above, gptblue] at (1.55,0.6) {$3$};

\node[below, gptorange] at (B2) {$4$};
\node[above, gptorange] at (3,0.6) {$4$};

\node[below, gptgold]   at (B7) {$9$};
\node[above, gptgold]   at (5.75,0.6) {$9$};

\end{tikzpicture}
}
\end{minipage}
\hfill
\begin{minipage}[c]{0.47\textwidth}
\centering
\resizebox{\linewidth}{!}{%
\begin{tikzpicture}[
    edge/.style       = {ultra thick}
]

\coordinate (A0)  at (-1.0,2.7) {};
\coordinate (A1)  at (1.5,3)  {};
\coordinate (A2)  at (4.5,3)  {};

\coordinate (B0)  at (-1.325,0.1) {};
\coordinate (B1)  at (0,0)      {};
\coordinate (B2)  at (1.5,0)    {};
\coordinate (B3)  at (3,0)      {};
\coordinate (B10) at (6,0)      {};

\draw[edge, black] ($(A0)+(-0.5,0)$) -- ($(B0)+(-0.2,0.3)$);

\draw[edge, black] ($(B0)+(0.1,0.3)$) -- ($(A1)+(-1,-0.50)$);
\draw[edge, black] ($(B0)+(0.1,0.3)$) -- ($(A2)+(-1.05,-0.60)$);

\draw[edge, gptblue]   ($(B1)+(-0.10,0.5)$) -- ($(A1)+(-1,-0.65)$);
\draw[edge, gptblue]   ($(B1)+( 0.15,0.5)$) -- ($(A2)+(-0.7,-0.8)$);
\draw[edge, gptblue] ($(A0)+(-0.3,-0.6)$) -- ($(B1)+(-0.10,0.40)$);

\draw[edge, gptorange] ($(B2)+(-0.10,0.50)$) -- ($(A1)+(-0.7,-0.9)$);
\draw[edge, gptorange] ($(B2)+( 0.0,0.50)$) -- ($(A2)+(-0.40,-1.0)$);
\draw[edge, gptorange] ($(A0)+( -0.2,-0.6)$) -- ($(B2)+(-0.2,0.50)$);

\draw[edge, gptpurple] ($(B3)+(0.10,0.50)$) -- ($(A1)+( -0.10,-1.10)$);
\draw[edge, gptpurple] ($(B3)+( 0.10,0.50)$) -- ($(A2)+( 0.00,-1.10)$);
\draw[edge, gptpurple] ($(A0)+( 0.1,-0.6)$) -- ($(B3)+(0,0.50)$);

\draw[edge, gptgold]   ($(B10)+(-0.10,0.50)$) -- ($(A1)+( 0.40,-1.0)$);
\draw[edge, gptgold]   ($(B10)+( 0.10,0.50)$) -- ($(A2)+( 1.10,-0.40)$);
\draw[edge, gptgold] ($(A0)+( 0.3,-0.5)$) -- ($(B10)+(-0.20,0.45)$);

\draw[fill=white] (A1) circle[radius=1.2];
\draw[fill=white] (A2) circle[radius=1.2];

\draw[fill=white] (A0) circle[radius=0.7];
\draw[fill=white] (B0) circle[radius=0.425];

\foreach \B in {B1,B2,B3,B10}
    \draw[fill=white] (\B) circle[radius=0.6];

\node at (4.5,0) {$\ldots$};

\draw[edge, black]         ($(A0)+(-0.35,0)$)  -- ($(A0)+(0.35,0)$);
\draw[edge, brilliantrose] ($(A1)+(-0.70,0)$)  -- ($(A1)+(0.70,0)$);
\draw[edge, applegreen]    ($(A2)+(-0.70,0)$)  -- ($(A2)+(0.70,0)$);

\draw[edge, black]         ($(B0)+(-0.20,0.1)$)  -- ($(B0)+(0.20,0.1)$);
\draw[edge, gptblue]       ($(B1)+(-0.40,0)$)  -- ($(B1)+(0.40,0)$);
\draw[edge, gptorange]     ($(B2)+(-0.40,0)$)  -- ($(B2)+(0.40,0)$);
\draw[edge, gptpurple]     ($(B3)+(-0.40,0)$)  -- ($(B3)+(0.40,0)$);
\draw[edge, gptgold]       ($(B10)+(-0.40,0)$) -- ($(B10)+(0.40,0)$);

\node[above right]       at (-0.9,3.3)  {$A_0$};
\node[above right]  at (2.3,3.9)    {$A_1$};
\node[above right] at (5.4,3.9)    {$A_2$};

\node[below right]       at (-1.45,-0.3) {$B_0$};
\node[below right] at (0.2,-0.5)  {$B_1$};
\node[below right] at (1.7,-0.5)  {$B_2$};
\node[below right] at (3.2,-0.5)  {$B_3$};
\node[below right] at (6.25,-0.5) {$B_{10}$};

\node[above, brilliantrose] at (A1)   {$1$};
\node[above, applegreen]    at (A2)   {$2$};

\node[below, gptblue]       at (B1)   {$3$};
\node[above, gptblue]       at ($(B1)+(0.19,0.52)$)   {$3$};

\node[below, gptorange]     at (B2)   {$4$};
\node[above, gptorange]     at ($(B2)+(0.03,0.57)$)   {$4$};

\node[below, gptpurple]     at (B3)   {$5$};
\node[above, gptpurple]     at ($(B3)+(0.05,0.57)$)   {$5$};

\node[below, gptgold]       at (B10)  {$12$};
\node[above, gptgold]       at ($(B10)+(-0.25,0.57)$) {$12$};

\node[above, black]         at (A0)   {$13$};
\node[below, black]         at ($(B0)+(0,0.1)$)   {$13$};
\node[above, black]         at ($(B0)+(0.03,0.37)$)   {$13$};

\end{tikzpicture}
}
\end{minipage}

\caption{Construction III from Conjecture~\ref{conjecture} for $10$  and $13$ colors (both with $a=2$).}
\label{Fig:constructionIII}

\end{figure}

We anticipate that for all $k$ the  best lower bounds presented above do determine the correct values of $\pi_k^{\min}(P_5)$.

\begin{conjecture}\label{conjecture} For every $k\geq4$, the value of $\pi_k^{\min}(P_5)$ is obtained by Construction~I for $k\in\{4,7\}$, by Construction II for $k\in\{5,6,8,9,11,12\}$, and by Construction III for $k=10$ and  $k \geq 13$.
\end{conjecture}

\section*{Generative AI statement}
Generative artificial intelligence tools were not used to obtain the results or to write this manuscript.

\end{document}